\documentclass[twoside,a4paper]{amsart}

\usepackage{amsfonts}
\usepackage[utf8]{inputenc}
\usepackage{amsthm}
\usepackage{verbatim}
\usepackage{amsmath, amssymb}
\usepackage{tikz}
\usetikzlibrary{decorations.pathreplacing}

\usepackage[titletoc,toc,title]{appendix}

\usepackage{listings}
\usepackage{color}
\usepackage{listings}
\usepackage[all]{xy}
\usepackage[pdftex,colorlinks,linkcolor=blue,citecolor=blue]{hyperref}
\usepackage{graphicx}
\usepackage{float}
\usepackage[margin=3cm]{geometry}
\usepackage{bigints}
\usepackage{dsfont}

\usepackage{mathrsfs}
\usepackage{mathtools}

\newcommand{\bN}{\mathbb{N}}

\newcommand{\bZ}{\mathbb{Z}}
\newcommand{\eps}{\varepsilon}

\renewcommand{\ge}{\geqslant}
\renewcommand{\le}{\leqslant}

\newcounter{dummy} \numberwithin{dummy}{section}

\theoremstyle{definition}

\newtheorem{prop}[dummy]{Proposition}

\newtheorem{thm}[dummy]{Theorem}
\newtheorem{lemma}[dummy]{Lemma}
\newtheorem{eg}[dummy]{Example}

\newtheorem{remark}[dummy]{Remark}

\newtheorem{question}[dummy]{Question}

\newtheorem*{thm*}{Theorem}
\newtheorem*{warning*}{Warning}

\newtheorem{thmx}{Theorem}
\title[On sets of popular differences in trees]{Direct and inverse results for popular differences in trees of positive dimension}
\date{\today}

\author[Alexander Fish and Leo Jiang]{Alexander Fish and Leo Jiang; \\ with a joint appendix with Ilya D. Shkredov} 
\address[Alexander Fish]{School of Mathematics and Statistics F07, University of Sydney, NSW 2006, Australia}
\email{alexander.fish@sydney.edu.au}
\address[Leo Jiang]{Department of Mathematics, University of Toronto, Toronto, ON M5S 2E4, Canada}
\email{ljiang@math.toronto.edu}
\address[Ilya D. Shkredov]{Steklov Mathematical Institute, ul. Gubkina, 8, Moscow, Russia, 119991,
and IITP RAS, Bolshoy Karetny per. 19, Moscow, Russia, 127994, and London Institute for Mathematical Sciences, 21 Albemarle St., UK}
\email{ilya.shkredov@gmail.com}

\keywords{Ramsey theory on trees, return times, inverse theorems, popular difference sets}

\begin{document}

\begin{abstract} 
We establish analogues for trees of results relating the density of a set $E \subset \mathbb{N}$, the density of its set of popular differences, and the structure of~$E$. 
To obtain our results, we formalise a correspondence principle of Furstenberg and Weiss which relates combinatorial data on a tree to the dynamics of a Markov process. 
Our main tools are Kneser-type inverse theorems for sets of return times in measure-preserving systems. 
In the ergodic setting we use a recent result of the first author with Bj\"orklund and Shkredov and a stability-type extension (proved jointly with Shkredov); we also prove a new result for non-ergodic systems. 
\end{abstract}

\maketitle

\section{Introduction}

In \cite{FW} Furstenberg and Weiss initiated the use of dynamical methods in the study of Ramsey theoretic questions for trees. 
They proved a Szemer\'{e}di-type theorem using a multiple recurrence result for a class of Markov processes (a purely combinatorial proof was later given by Pach, Solymosi, and Tardos \cite{PST}). 
More precisely, they showed that finite replicas of the full binary tree could always be found in (infinite) trees of positive growth rate. 
It is then a natural question to quantify the abundance of finite configurations in a tree in relation to its size as measured by its upper Minkowski and Hausdorff dimensions. 

To begin, we review the analogous question in the integer setting. 
Specifically, we consider the abundance of configurations in a subset $E \subset \mathbb{N}$.
Recall that the \emph{upper density} and \emph{upper Banach density} of $E$ are 
\[\overline{d}(E) = \limsup_{N \to \infty} \frac{|E \cap \{0,\ldots,N\}|}{N+1}, \quad d^\ast(E) = \limsup_{N-M \to \infty} \frac{|E \cap \{M,\ldots,N\}|}{N-M+1}.\]
The abundance of $2$-term arithmetic progressions in $E$ can be related to the density of $E$ in the following way. 
Consider the sets of \emph{popular differences of $E$} with respect to $\overline{d}$ and $d^\ast$ defined by 
 \[
\overline{\Delta}_0(E) = \{n \in \mathbb{N} \colon \overline{d}(E \cap (E-n))>0\}, \quad \Delta_0^\ast(E) = \{n \in \mathbb{N} \colon d^\ast(E \cap (E-n))>0\}.
\]
Furstenberg's correspondence principle \cite{Fu0} states that there exists a measure-preserving system $(X,\mathscr{B},\nu,S)$ and $A \in \mathscr{B}$ with $\nu(A) = \overline{d}(E)$ such that for all integers $k \ge 1$ and $0=n_1,\ldots,n_k \in \mathbb{N}$, 
\[\overline{d}\left((E-n_1) \cap \cdots \cap (E-n_k)\right) \ge \nu\left(S^{-n_1}A \cap \cdots \cap S^{-n_k}A\right).\]
Taking $k=2$, it follows that $\overline{\Delta}_0(E)$ contains 
\[
\mathcal{R}=\mathcal{R}(A) = \{n \in \mathbb{N} \colon \nu(A \cap S^{-n}A)>0\},
\]
the set of \emph{return times of $A$}. 
Applying the mean ergodic theorem then gives 
\begin{equation}\label{pop_diff}
\underline{d}(\overline{\Delta}_0(E)) \ge \underline{d}(\mathcal{R}) \ge \lim_{N \to \infty} \frac{1}{N+1} \sum_{n=0}^{N}\frac{\nu(A \cap S^{-n}A)}{\nu(A)} \ge \nu(A) = \overline{d}(E),
\end{equation}
where the lower density $\underline{d}$ is defined for $E \subset \mathbb{N}$ by
\[\underline{d}(E) = \liminf_{N \to \infty} \frac{|E \cap \{0,\ldots,N\}|}{N+1}.\]
If in the above the upper density is replaced by the upper Banach density then $\nu$ can further be chosen to be ergodic \cite[Proposition 3.9]{Fur4} (see \cite[Proposition 3.1]{BHK} for an explicit proof). 

Following Furstenberg and Weiss \cite{FW}, we formulate a correspondence principle for arbitrary finite configurations in a tree and use it to obtain analogues of the inequality~(\ref{pop_diff}). 
We then analyse the case of equality in~(\ref{pop_diff}) and its analogues for trees using inverse theorems for the set of return times. 
In the ergodic situation we use a result of Bj\"{o}rklund, the first author, and Shkredov \cite{BFS} and a stability-type extension proved jointly with Shkredov in Appendix \ref{appendix}, while in the general case we prove a slightly weaker statement (Theorem~\ref{first_inverse_thm_measure}). 
Using these we obtain inverse theorems for inequality~(\ref{pop_diff}): a tree for which equality holds must contain arbitrarily long ``arithmetic progressions'' with a fixed common difference. 

\subsection{Main Results}

To describe our results, we first summarise the necessary definitions (see Section \ref{Markov_processes} for precise formulations). 
For clarity of exposition, in this introduction we restrict our attention to the case $r=2$ of our results and make corresponding simplifications to the notation.

Fix an integer $q \ge 2$. 
In this paper a tree can be visualised as a directed graph $T$ with a distinguished vertex (the root) having no incoming edges, such that each vertex has between $1$ and $q$ outgoing edges and each nonroot vertex has exactly one incoming edge. 
(Technically, we work with the vertices of the graph with the partial order induced by directed paths.)  
The ``size'' of $T$ can be quantified by its upper Minkowski and Hausdorff dimensions $\overline{\dim}_MT$ and $\dim T$, which are defined by an identification of such trees with closed subsets of $[0,1]$. 

\subsubsection{Tree analogues of popular difference sets}

A $k$-term arithmetic progression ($k$-AP) in $E \subset \mathbb{N}$ can be viewed as an affine map $\{0,\ldots,k-1\} \to E$. 
We consider ``affine'' maps satisfying certain branching conditions from configurations~$C$ (``finite trees'') to trees $T$. 
If there exists such a map with ``common difference'' $n$ taking the root of the configuration to $v \in T$, we say that $v \in C_n = C_n(T)$. 
The set $C_n$ corresponds to the set $E \cap (E-n) \cap \cdots \cap (E-(k-1)n)$ for $k$-APs in $E \subset \mathbb{N}$. 
Using extensions of upper density and upper Banach density to subsets of trees, we define sets of ``generic parameters''
\[\overline{G}(C) = \{n \in \mathbb{N} \colon \overline{d}(C_n) > 0 \}, \qquad G^\ast(C) = \{n \in \mathbb{N} \colon d^\ast(C_n) > 0 \}.\]
We also introduce certain configurations $F$ and $D$ which are analogues of $2$-APs, and their generic parameters can be interpreted as popular differences for trees. 
In particular, our first result is a version of (\ref{pop_diff}):

\begin{thmx}[= Theorem \ref{Theorem4.1} and Theorem \ref{direct-D-upper} for $r=2$]
\label{first_direct_thm}
	For any tree $T$ we have 
	\[\underline{d}(\overline{G}(F)) \ge \underline{d}(\overline{G}(D)) \ge \overline{\dim}_M T \quad \text{ and } \quad \underline{d}(G^\ast(F)) \ge \underline{d}(G^\ast(D)) \ge \dim T.\] 
\end{thmx}

\subsubsection{Inverse theorems for sets of return times}
\label{subsection_inverse_thms_ergodic}

Given the direct result Theorem~\ref{first_direct_thm}, we are interested in characterising trees such that equality holds (or almost holds). 
To illustrate the ideas we consider here the situation when equality is (almost) achieved in (\ref{pop_diff}), which is the analogous question for subsets of $\mathbb{N}$. 
Observe that the density of the set of return times of $A$ is then close to the measure of $A$. 
It is natural to expect in this situation that the dynamics of $A$ under $S$ is rigid in some way, and this is indeed the case. 

Let $(X,\mathscr{B},\nu,S)$ be a measure-preserving system, and let $A$ be a measurable set with $\nu(A) > 0$ and set of return times $\mathcal{R}$. 
Using a theorem of Kneser we prove the following result: 

\begin{thmx}[= Theorem \ref{first_inverse_thm_measure}]\label{ThmB}
If $\overline{d}(\mathcal{R}) = \nu(A) > 0$, then there exists an integer $m \ge 1$ such that up to $\nu$-null sets
\[X = \bigsqcup_{i=0}^{m-1} S^{-i} A. \]
\end{thmx} 

\begin{question}
\label{q1}
Does the assumption $\underline{d}(\mathcal{R}) = \nu(A)$ suffice to prove the conclusion of Theorem~\ref{ThmB}? 
\end{question}

If $\nu$ is ergodic then Question \ref{q1} has an affirmative answer, and further there is an inverse result for cases of almost equality. 
The following theorem is an easy corollary of results by Bj\"orklund, the first author, and Shkredov in \cite{BFS}: 

\begin{thm}[= Theorem \ref{second_inverse_thm_measure}]\label{Thm5.4}
If $(X,\mathscr{B},\nu,S)$ is ergodic and 
\[0 < \underline{d}(\mathcal{R}) < \frac{3}{2} \nu(A),\]
then there exists an integer $m \ge 1$ such that $\mathcal{R} = m \bN$ and $X = \bigsqcup_{i=0}^{m-1} S^{-i}\left( \bigcup_{j=0}^{\infty}S^{-jm} A \right)$ up to $\nu$-null sets. 
\end{thm}

\begin{remark}
Example 1.2 in \cite{BFS} shows that for every $\beta > 1$ there exists a non-ergodic measure-preserving system $(X,\mathscr{B},\nu,S)$ and $A \in \mathscr{B}$ of arbitrarily small measure such that $\overline{d}(\mathcal{R}) \le \beta \nu(A)$ and there is no $m \ge 1$ such that $\mathcal{R} = m\mathbb{N}$. 
\end{remark}

\subsubsection{Inverse results for popular difference sets}

As a corollary of Theorem \ref{ThmB} and Furstenberg's correspondence principle we immediately obtain the following inverse-type result for (\ref{pop_diff}): 
\begin{prop}
\label{first_inverse_thm_sets}
Assume that $E \subset \mathbb{N}$ satisfies $\overline{d}(\overline{\Delta}_0(E)) = \overline{d}(E) >0$. 
Then there exists $m \ge 1$ such that $m \mathbb{N} \subset \overline{\Delta}_0(E)$ and $\overline{d}(\overline{\Delta}_0(E)) = \overline{d}(E) = m^{-1}$. 
Moreover, for every $k \ge 2$ 
\[
\overline{d}\left( E \cap (E-m) \cap \ldots \cap (E- (k-1)m) \right)  = \overline{d}(E). 
\]
\end{prop}

If we consider $\Delta^\ast_0(E)$ and $d^\ast(E)$ in place of $\overline{\Delta}_0(E)$ and $\overline{d}(E)$, we can apply Theorem \ref{Thm5.4} to obtain the following inverse result:

\begin{prop}
\label{second_inverse_thm_sets}
Let $1 \le \beta < 3/2$. Assume that $E \subset \mathbb{N}$ satisfies
\[0 < \underline{d}(\Delta^\ast_0(E)) = \beta \cdot d^\ast(E). \]
Then there exists $m \ge 1$ such that $m \mathbb{N} \subset \Delta^\ast_0(E)$.
Moreover, for every $k \ge 2$ that satisfies $(1 - \beta^{-1})k < 1$
we have 
\[
d^\ast\left(E \cap (E-m) \cap \ldots \cap (E- (k-1)m)\right) > 0. 
\]
\end{prop}

\subsubsection{Inverse results for $\overline{G}(F)$ and $G^\ast(F)$}

Propositions \ref{first_inverse_thm_sets} and \ref{second_inverse_thm_sets} can be interpreted as saying that (almost) equality holds in (\ref{pop_diff}) for a subset $E \subset \mathbb{N}$ only if $E$ is ``similar'' to the periodic set $m\mathbb{N}$. 
In the tree setting we prove analogous results. 

For every $m \ge 1$, define $T_{m\mathbb{N}}$ to be the tree such that $v \in T_{m\mathbb{N}}$ has $q$ outgoing edges if the directed path from the root to $v$ has length a multiple of $m$, and $1$ outgoing edge otherwise. 
The inequalities in Theorem \ref{first_direct_thm} are equalities for $T_{m\mathbb{N}}$ (see Subsection \ref{equality_case}).

For every $k \ge 1$, define the configuration $V^{m,k}$ to be the first $k$ levels of $T_{m\mathbb{N}}$. 
The following two theorems are analogues of Proposition \ref{first_inverse_thm_sets} and Proposition~\ref{second_inverse_thm_sets} respectively.
\begin{thmx}[= Theorem \ref{first_inverse_thm_trees} for $r=2$]
Let $T$ be a tree.
Assume that 
\[\overline{d}(\overline{G}(F)) = \overline{\dim}_M T  > 0.
\]
Then there exists an integer $m \ge 1$ such that $\overline{\dim}_M T = m^{-1}$, and $\overline{d}(V^{m,k}) > 0$ for every $k \ge 1$.
\end{thmx}

\begin{thmx}[= Theorem \ref{second_inverse_thm_trees} for $r=2$]\label{ThmD}
		Let $T$ be a tree. 
Assume that 
\[
\underline{d}(G^\ast(F)) = \dim T > 0 \quad \text{or} \quad \overline{d}(G^\ast(D)) = \dim T > 0. 
\]
	Then there exists an integer $m \ge 1$ such that $\dim T = m^{-1}$, and $d^\ast(V^{m,k}) > 0$ for every $k \ge 1$.
\end{thmx}

\begin{remark} 
We show in Subsection \ref{sharpness} that Theorem \ref{ThmD} cannot be improved. 
Indeed, for every $\eps > 0$ there exists a tree $T_{\eps}$ such that 
\[
0 < \dim T_\eps \le \underline{d}(G^\ast(F)) < (1+\eps) \dim T_\eps
\]
and the configuration $V^{m,k}$ does not appear at all in $T_{\eps}$ for some large $k$.
\end{remark}

Our final result is another partial analogue of Proposition \ref{second_inverse_thm_sets}. 

\begin{thm}[= Theorem \ref{third_inverse_thm_trees} for $r=2$]
Let $T$ be a tree. 
Assume that there exists  $\beta < 3/2$ such that
\[
0 < \underline{d}(G^\ast(F)) = \beta \cdot \dim T. 
\]
Then there exists $m \ge 1$ with $m \bN \subset G^\ast(F)$. 
\end{thm}

\subsection*{Organisation of the paper}
After describing the combinatorial and dynamical background (Section \ref{Markov_processes}) and establishing Furstenberg--Weiss correspondence principles (Section \ref{Correspondence_principle}), in Section \ref{Direct_Theorems} we prove lower bounds for the densities of popular differences for trees. 
We then use inverse theorems for sets of return times in measure-preserving systems (Section \ref{Inverse_theorems_measure} and Appendix \ref{appendix}) to prove inverse theorems for these lower bounds (Section \ref{inverse_theorems_trees}). 

\subsection*{Acknowledgments} We thank Michael Bj\"orklund and James Parkinson. The current paper is a sequel of joint works of A.F. and I.S. with them. A.F. is grateful to Itai Benjamini for inspiring conversations on the subject during his visit to the Weizmann Institute hosted by Omri Sarig. He thanks Omri and the Weizmann Institute for hospitality. He also thanks Haotian Wu and Cecilia Gonz\'alez Tokman for fruitful discussions.
I.S. is grateful to SMRI and the School of Mathematics and Statistics at the University of Sydney for funding his visit and for their hospitality.
We also thank the anonymous referee for a thorough reading of the paper, which led to many improvements.

\section{Trees and Markov processes}
\label{Markov_processes}

Fix an integer $q \ge 2$, and for $2 \le r \le q$ define $\Lambda_r = \{0,\ldots,r-1\}$ and $\Lambda = \Lambda_q$. 
We set $\mathbb{N} = \{0,1,\ldots\}$. 

\subsection*{Combinatorial setup}

Let $\Lambda^\ast = \cup_{n=0}^\infty \Lambda^n$ be the set of finite words over $\Lambda$, where $\Lambda^0$ is the singleton comprising the empty word~$\emptyset$. 
Consider the partial order $\le$ on $\Lambda^\ast$ defined by $v \le w$ if $w$ is the concatenation $vu$ of $v$ and some $u \in \Lambda^\ast$. 
A \emph{tree} is then a nonempty subset $T \subset \Lambda^\ast$ closed under predecessors and having no maximal elements with respect to $\le$. 
We refer to elements of $T$ as \emph{vertices} (using the natural graph-theoretic terminology), and write $l(v)=n$ if $v \in T(n) = T \cap \Lambda^n$. 
Every tree contains $\emptyset$ (the \emph{root}), and for every $v \in T$ there is a tree $T^v = \{w \in \Lambda^\ast \colon vw \in T\}$. 

\begin{remark}
Trees are combinatorial realisations of closed sets in $\Lambda^\mathbb{N}$, a symbolic analogue of $[0,1]$. 
Given a tree $T$, the set 
\[\{(a_i)_{i\ge 0} \in \Lambda^\mathbb{N} \colon (a_0,\ldots,a_n) \in T \text{ for all }n \in \mathbb{N}\}\]
is closed in $\Lambda^\mathbb{N}$ (with the product of discrete topologies on $\Lambda$), and there is an inverse map sending a closed subset $A \subset \Lambda^\mathbb{N}$ to the tree 
\[\{v \in \Lambda^\ast \colon vw \in A \text{ for some }w \in \Lambda^\mathbb{N}\}.\] 
This motivates several definitions we give below. 
\end{remark}

The \emph{(upper) Minkowski dimension} of $T$ is 
\[ \overline{\dim}_M T = \limsup_{N \to \infty} \frac{\log_q|T(N)|}{N}.\]
To define the Hausdorff dimension of a tree, we first define the analogue of an irredundant open cover for trees. 
A \emph{section} of a tree $T$ is a finite subset $\Pi \subset T$ such that $|\Pi \cap \{w \in T \colon w \le v\}|=1$ for all but finitely many $v \in T$. 
Define also $l(\Pi) = \min \{l(v) \colon v \in \Pi\}$. 
Then the \emph{Hausdorff dimension} of $T$ is
\[\dim T = \inf \left\{ \lambda>0 \colon \liminf_{N \to \infty} \inf_{\substack{l(\Pi)=N \\ \Pi \text{ section of } T}} \sum_{v \in \Pi} q^{-\lambda l(v)}  < 1\right\}.\]

\begin{eg}\label{Ex.2.2}
Given $E \subset \mathbb{N}$ and $2 \le r \le q$, define the tree
\[
T_E^r = \{\emptyset\} \cup \bigcup_{i=0}^\infty \prod_{0 \le j \le i} \Gamma_j, \quad \text{ where } \Gamma_j = \begin{cases} \Lambda & \text{ if } j \in E \\ \Lambda_{r-1} & \text{ otherwise.} \end{cases}
\]
A straightforward calculation shows that  
	\begin{align*}
		\overline{\dim}_M T_E^r &= \limsup_{N \to \infty} \frac{\log_q { q^{|E \cap \{0,\ldots,N-1\}|} (r-1)^{|E^c \cap \{0,\ldots,N -1\}|}}}{N}
\\
&= \overline{d}(E) + \log_q (r-1) (1 - \overline{d}(E)).
	\end{align*}
If $E$ is a ``periodic'' set (such as $m\mathbb{N}$) then $T_E^r$ is ``self-similar'' and $\overline{\dim}_M T_E^r = \dim T_E^r$. 
\end{eg}

Elements of $\Lambda^\ast$ correspond to cylinder sets of $\Lambda^\mathbb{N}$. 
By the Carath\'{e}odory extension theorem, Borel probability measures on $\Lambda^\mathbb{N}$ are in bijection with functions $\tau \colon \Lambda^\ast \to [0,1]$ such that $\tau(\emptyset) = 1$ and $\tau(v) = \sum_{a \in \Lambda} \tau(va)$ for all $v \in \Lambda^\ast$. 
We call such functions \emph{Markov trees}, since the support $|\tau| = \{v \in \Lambda^\ast \colon \tau(v)>0\}$ of such a function is a tree. 
The set of Markov trees is a closed subspace of the compact space $[0,1]^{\Lambda^\ast}$ with metric $d(\tau_1,\tau_2) = \sum_{v \in \Lambda^\ast} q^{-l(v)}|\tau_1(v)-\tau_2(v)|$. 
By abuse of notation we denote it by $\mathcal{P}(\Lambda^\mathbb{N})$, since it is homeomorphic to the space of Borel probability measures on $\Lambda^\mathbb{N}$ with the weak-${}^\ast$ topology. 

The \emph{dimension} of a Markov tree \cite[Definition 7]{Fur1} is
\[\dim \tau = \liminf_{\substack{l(\Pi)\to\infty \\ \Pi \text{ section of } |\tau|}} \frac{-\sum_{v \in \Pi} \tau(v)\log_q \tau(v)}{\sum_{v \in \Pi}l(v)\tau(v)}.\]

Given a subset $V \subset T$ we define its \emph{upper density} 
\[ \overline{d}(V) = \limsup_{N \to \infty} \frac{1}{|T(N)|} \sum_{v \in T(N)} \frac{|V \cap \{w \in T \colon w \le v\}|}{N+1} \]
and its \emph{upper Banach density}
\[d^\ast(V) =  \limsup_{N \to \infty} \sup_{\substack{|\tau| \subset T \\ v \in |\tau|}} \frac{1}{N+1} \sum_{l(w) \le N} \frac{\tau(vw)}{\tau(v)}1_V(vw).\]

\begin{remark}
These definitions specialise to their integer counterparts in the degenerate case $q=1$, justifying the notation. 
The inequality $d^\ast(V) \ge \overline{d}(V)$ also holds for our more general definition. 
To see this, observe that it is enough to construct Markov trees $\pi_N$ supported on $T$ such that
\[\sum_{l(w) \le N} \pi_N(w)1_V(w) = \sum_{v \in T(N)} \frac{|V \cap \{w \in T \colon w \le v\}|}{|T(N)|} = \sum_{l(w) \le N} \frac{|\{v \in T(N) \colon w \le v\}|}{|T(N)|}1_V(w)\]
(the last equality follows from reindexing the sum). 
But the above formula defines such a Markov tree on vertices $w$ with $l(w) \le N$, and we can choose $\pi_N$ to be any consistent extension to the remaining vertices (cf. the proof of Theorem \ref{FW-correspondence}). 
\end{remark}

\begin{eg}\label{Ex.2.4}
If $V = V(E) = \{v \in T \colon l(v) \in E\}$ for $E \subset \mathbb{N}$ and $T$ a tree, then $\overline{d}(V) = \overline{d}(E)$ and $d^\ast(V) = d^\ast(E)$. 
Both equalities follow directly from the definitions. 
For example, for the second equality we observe that for any $\tau$ with $|\tau| \subset T$ and any $v \in |\tau|$ we have 
	\[\frac{1}{N+1}\sum_{l(w) \le N} \frac{\tau(vw)}{\tau(v)}1_V(vw) = \frac{|E \cap \{l(v),\ldots, l(v)+N\}|}{N+1}.\] 
\end{eg}

We use the term \emph{configuration} to refer to a nonempty finite subset $C \subset \Lambda^\ast$ closed under predecessors (a finite tree). 
Terminology and notation defined above for trees are used for configurations as appropriate without comment. 
A configuration $C$ is \emph{nonbranching} if $|C(n)| \le 1$ for all $n \in \mathbb{N}$ and \emph{branching} otherwise. 

By analogy with arithmetic progressions in $\mathbb{N}$, we consider ``affine embeddings'' of $C$ in a tree $T$. 
More precisely, for a vertex $v \in T$ and $n \in \mathbb{N}$ we say $v \in C_n = C_n(T)$ if there exists a map $\iota \colon C \to T$ such that 
\begin{itemize}
\item $\iota(\emptyset)=v$,
\item $\iota(w_1) \le \iota(w_2)$ if $w_1 \le w_2$ ($\iota$ is a map of posets),
\item if $w$ is the longest initial subword common to $w_1$ and $w_2$, then $\iota(w)$ is the longest initial subword common to $\iota(w_1)$ and $\iota(w_2)$ ($\iota$ is infimum-preserving),  
\item $l(\iota(w)) = l(v)+nl(w)$ for all $w \in C$ ($\iota$ is ``affine'').
\end{itemize}
Equivalently, we say the configuration $C$ appears at $v$ (with parameter $n$). 
Observe that trivially every configuration appears at every vertex with parameter $0$. 

We will be concerned with the following configurations (see Figure \ref{config_example}): 
\[F^r = \{\emptyset\} \cup \Lambda_r \cup 0\Lambda_r, \mbox{  } D^{r,k} = \bigcup_{n=0}^k \Lambda_r^n, \mbox{  } V^{r,m,k} = \{v \in T_{m\mathbb{N}}^r \colon l(v) \le k+1\}.\]

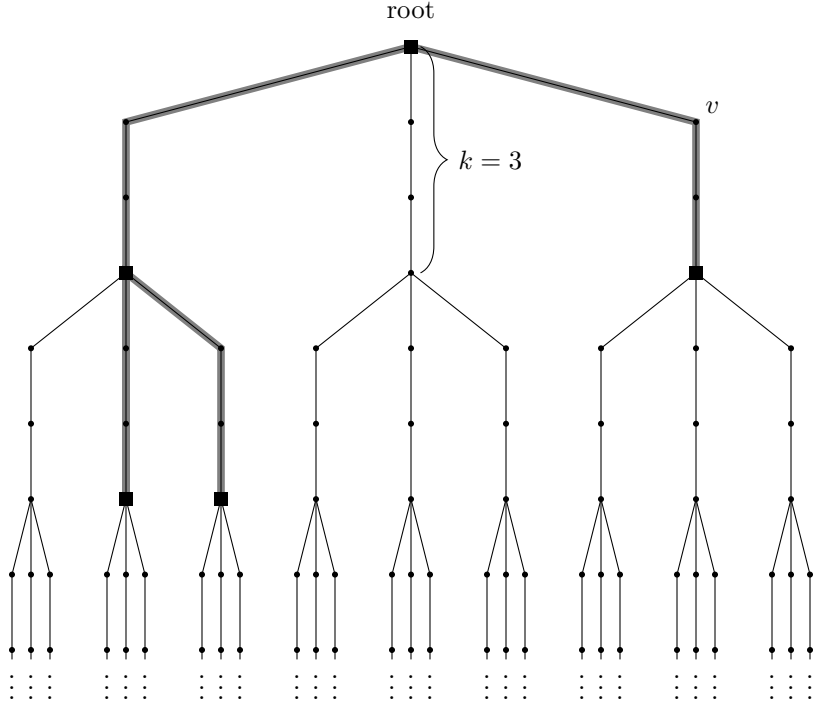
\begin{figure}

\centering
\begin{tikzpicture}[scale=0.25]

\def\dotwidth{0.8pt}
\def\configwidth{2.5pt}

\draw[gray,line width=1mm] (0,20) -- (-15,16) -- (-15,-4) -- (-15,8) -- (-10,4) -- (-10,-4);
\draw[gray,line width=1mm] (0,20) -- (15,16) -- (15,8);

\node[above] at (0,21) {root};

\draw (0,20) -- (0,-12.5);
\draw (0,20) -- (-15,16) -- (-15,-12.5);
\draw (0,20) -- (15,16) -- (15,-12.5);
\draw (-15,8) -- (-20,4) -- (-20,-12.5);
\draw (-15,8) -- (-10,4) -- (-10,-12.5);
\draw (0,8) -- (-5,4) -- (-5,-12.5);
\draw (0,8) -- (5,4) -- (5,-12.5);
\draw (15,8) -- (10,4) -- (10,-12.5);
\draw (15,8) -- (20,4) -- (20,-12.5);
\draw (-20,-4) -- (-21,-8) -- (-21,-12.5);
\draw (-20,-4) -- (-19,-8) -- (-19,-12.5);
\draw (-15,-4) -- (-16,-8) -- (-16,-12.5);
\draw (-15,-4) -- (-14,-8) -- (-14,-12.5);
\draw (-10,-4) -- (-11,-8) -- (-11,-12.5);
\draw (-10,-4) -- (-9,-8) -- (-9,-12.5);
\draw (-5,-4) -- (-6,-8) -- (-6,-12.5);
\draw (-5,-4) -- (-4,-8) -- (-4,-12.5);
\draw (0,-4) -- (-1,-8) -- (-1,-12.5);
\draw (0,-4) -- (1,-8) -- (1,-12.5);
\draw (5,-4) -- (4,-8) -- (4,-12.5);
\draw (5,-4) -- (6,-8) -- (6,-12.5);
\draw (10,-4) -- (9,-8) -- (9,-12.5);
\draw (10,-4) -- (11,-8) -- (11,-12.5);
\draw (15,-4) -- (14,-8) -- (14,-12.5);
\draw (15,-4) -- (16,-8) -- (16,-12.5);
\draw (20,-4) -- (19,-8) -- (19,-12.5);
\draw (20,-4) -- (21,-8) -- (21,-12.5);

\node at (0,20) [circle,fill,inner sep=\dotwidth]{};
\node at (0,16) [circle,fill,inner sep=\dotwidth]{};
\node at (0,12) [circle,fill,inner sep=\dotwidth]{};
\node at (0,8) [circle,fill,inner sep=\dotwidth]{};
\node at (0,4) [circle,fill,inner sep=\dotwidth]{};
\node at (0,0) [circle,fill,inner sep=\dotwidth]{};
\node at (0,-4) [circle,fill,inner sep=\dotwidth]{};
\node at (0,-8) [circle,fill,inner sep=\dotwidth]{};
\node at (0,-12) [circle,fill,inner sep=\dotwidth]{};
\node at (-15,16) [circle,fill,inner sep=\dotwidth]{};
\node at (-15,12) [circle,fill,inner sep=\dotwidth]{};
\node at (-15,8) [circle,fill,inner sep=\dotwidth]{};
\node at (-15,4) [circle,fill,inner sep=\dotwidth]{};
\node at (-15,0) [circle,fill,inner sep=\dotwidth]{};
\node at (-15,-4) [circle,fill,inner sep=\dotwidth]{};
\node at (-15,-8) [circle,fill,inner sep=\dotwidth]{};
\node at (-15,-12) [circle,fill,inner sep=\dotwidth]{};
\node at (15,16) [circle,fill,inner sep=\dotwidth]{};
\node at (15,12) [circle,fill,inner sep=\dotwidth]{};
\node at (15,8) [circle,fill,inner sep=\dotwidth]{};
\node at (15,4) [circle,fill,inner sep=\dotwidth]{};
\node at (15,0) [circle,fill,inner sep=\dotwidth]{};
\node at (15,-4) [circle,fill,inner sep=\dotwidth]{};
\node at (15,-8) [circle,fill,inner sep=\dotwidth]{};
\node at (15,-12) [circle,fill,inner sep=\dotwidth]{};
\node at (-5,4) [circle,fill,inner sep=\dotwidth]{};
\node at (-5,0) [circle,fill,inner sep=\dotwidth]{};
\node at (-5,-4) [circle,fill,inner sep=\dotwidth]{};
\node at (-5,-8) [circle,fill,inner sep=\dotwidth]{};
\node at (-5,-12) [circle,fill,inner sep=\dotwidth]{};
\node at (5,4) [circle,fill,inner sep=\dotwidth]{};
\node at (5,0) [circle,fill,inner sep=\dotwidth]{};
\node at (5,-4) [circle,fill,inner sep=\dotwidth]{};
\node at (5,-8) [circle,fill,inner sep=\dotwidth]{};
\node at (5,-12) [circle,fill,inner sep=\dotwidth]{};
\node at (20,4) [circle,fill,inner sep=\dotwidth]{};
\node at (20,0) [circle,fill,inner sep=\dotwidth]{};
\node at (20,-4) [circle,fill,inner sep=\dotwidth]{};
\node at (20,-8) [circle,fill,inner sep=\dotwidth]{};
\node at (20,-12) [circle,fill,inner sep=\dotwidth]{};
\node at (-20,4) [circle,fill,inner sep=\dotwidth]{};
\node at (-20,0) [circle,fill,inner sep=\dotwidth]{};
\node at (-20,-4) [circle,fill,inner sep=\dotwidth]{};
\node at (-20,-8) [circle,fill,inner sep=\dotwidth]{};
\node at (-20,-12) [circle,fill,inner sep=\dotwidth]{};
\node at (10,4) [circle,fill,inner sep=\dotwidth]{};
\node at (10,0) [circle,fill,inner sep=\dotwidth]{};
\node at (10,-4) [circle,fill,inner sep=\dotwidth]{};
\node at (10,-8) [circle,fill,inner sep=\dotwidth]{};
\node at (10,-12) [circle,fill,inner sep=\dotwidth]{};
\node at (-10,4) [circle,fill,inner sep=\dotwidth]{};
\node at (-10,0) [circle,fill,inner sep=\dotwidth]{};
\node at (-10,-4) [circle,fill,inner sep=\dotwidth]{};
\node at (-10,-8) [circle,fill,inner sep=\dotwidth]{};
\node at (-10,-12) [circle,fill,inner sep=\dotwidth]{};

\node at (-21,-8) [circle,fill,inner sep=\dotwidth]{};
\node at (-21,-12) [circle,fill,inner sep=\dotwidth]{};
\node at (-19,-8) [circle,fill,inner sep=\dotwidth]{};
\node at (-19,-12) [circle,fill,inner sep=\dotwidth]{};
\node at (-16,-8) [circle,fill,inner sep=\dotwidth]{};
\node at (-16,-12) [circle,fill,inner sep=\dotwidth]{};
\node at (-14,-8) [circle,fill,inner sep=\dotwidth]{};
\node at (-14,-12) [circle,fill,inner sep=\dotwidth]{};
\node at (-11,-8) [circle,fill,inner sep=\dotwidth]{};
\node at (-11,-12) [circle,fill,inner sep=\dotwidth]{};
\node at (-9,-8) [circle,fill,inner sep=\dotwidth]{};
\node at (-9,-12) [circle,fill,inner sep=\dotwidth]{};
\node at (-6,-8) [circle,fill,inner sep=\dotwidth]{};
\node at (-6,-12) [circle,fill,inner sep=\dotwidth]{};
\node at (-4,-8) [circle,fill,inner sep=\dotwidth]{};
\node at (-4,-12) [circle,fill,inner sep=\dotwidth]{};
\node at (-1,-8) [circle,fill,inner sep=\dotwidth]{};
\node at (-1,-12) [circle,fill,inner sep=\dotwidth]{};

\node at (21,-8) [circle,fill,inner sep=\dotwidth]{};
\node at (21,-12) [circle,fill,inner sep=\dotwidth]{};
\node at (19,-8) [circle,fill,inner sep=\dotwidth]{};
\node at (19,-12) [circle,fill,inner sep=\dotwidth]{};
\node at (16,-8) [circle,fill,inner sep=\dotwidth]{};
\node at (16,-12) [circle,fill,inner sep=\dotwidth]{};
\node at (14,-8) [circle,fill,inner sep=\dotwidth]{};
\node at (14,-12) [circle,fill,inner sep=\dotwidth]{};
\node at (11,-8) [circle,fill,inner sep=\dotwidth]{};
\node at (11,-12) [circle,fill,inner sep=\dotwidth]{};
\node at (9,-8) [circle,fill,inner sep=\dotwidth]{};
\node at (9,-12) [circle,fill,inner sep=\dotwidth]{};
\node at (6,-8) [circle,fill,inner sep=\dotwidth]{};
\node at (6,-12) [circle,fill,inner sep=\dotwidth]{};
\node at (4,-8) [circle,fill,inner sep=\dotwidth]{};
\node at (4,-12) [circle,fill,inner sep=\dotwidth]{};
\node at (1,-8) [circle,fill,inner sep=\dotwidth]{};
\node at (1,-12) [circle,fill,inner sep=\dotwidth]{};

\foreach \i in {-21,-20,-19,-16,-15,-14,-11,-10,-9,-6,-5,-4,-1,0,1,4,5,6,9,10,11,14,15,16,19,20,21}
{\node[below] at (\i,-12) {$\vdots$};};

\node at (0,20) [rectangle, fill,inner sep=\configwidth]{};
\node at (-15,8) [rectangle,fill,inner sep=\configwidth]{};
\node at (15,8) [rectangle,fill,inner sep=\configwidth]{};
\node at (-15,-4) [rectangle,fill,inner sep=\configwidth]{};
\node at (-10,-4) [rectangle,fill,inner sep=\configwidth]{};

\draw[decorate, decoration={brace,amplitude=10pt}] (0.5,20) -- (0.5,8);
\node[right] at (2,14) {$m=3$};

\node[above right] at (15,16) {$v$};

\end{tikzpicture}

\caption{The configuration $F^2$ appears at the root of $T^2_{3\mathbb{N}}$ with parameter $n=3$, while $v \notin F^2_n$ for any $n \ge 1$.}
\label{config_example}
\end{figure}

For every configuration $C$ and tree $T$ we define the sets of generic parameters 
\[\overline{G}(C) =\overline{G}(C,T) = \{n \in \mathbb{N} \colon \overline{d}(C_n)>0\},\]
\[G^\ast(C) = G^\ast(C,T) = \{n \in \mathbb{N} \colon d^\ast(C_n)>0\}.\]

\begin{remark}\label{Rmk.2.5}
Notice that $F^r$ appears at $v \in T_E^r$ with parameter $n$ if and only if $D^{r,2}$ appears at $v$ with parameter $n$ if and only if $l(v), l(v)+n \in E$. 
Therefore $\overline{G}(F^r,T^r_E) = \overline{G}(D^{r,2},T^r_E) = \overline{\Delta_0}(E)$ and $G^\ast(F^r,T^r_E)= G^\ast(D^{r,2},T^r_E) = \Delta^{\ast}_0(E)$ by Example \ref{Ex.2.4}.
	This is why the generic parameters of $F^r$ and $D^{r,2}$ are analogues of popular differences for trees. 
\end{remark}

\subsubsection{Equality in Theorems \ref{Theorem4.1} and  \ref{direct-D-upper}} 
\label{equality_case}
The tree $T^r_{m\mathbb{N}}$ achieves equality in Theorems \ref{Theorem4.1} and \ref{direct-D-upper}. 
Indeed, by Example \ref{Ex.2.2} 
\[ \overline{\dim}_M T_{m\mathbb{N}}^r =  \frac{1}{m} + \log_q (r-1) \left(1 - \frac{1}{m}\right). \]
The self-similarity of $T_{m\mathbb{N}}^r$ implies that $\dim T_{m\mathbb{N}}^r = \overline{\dim}_M T_{m\mathbb{N}}^r$. 
Also, by Remark \ref{Rmk.2.5} it follows that
\[ \overline{G}(F^r, T^r_{m\mathbb{N}}) = G^\ast(F^r,T_{m\mathbb{N}}^r) = m \mathbb{N} \quad \text{and} \quad \underline{d}(\overline{G}(F^r,T^r_{m\mathbb{N}})) = \underline{d}(G^\ast(F^r,T_{m\mathbb{N}}^r)) = m^{-1}. \]
Hence
\[ \underline{d}(\overline{G}(F^r,T_{m\mathbb{N}}^r)) = \frac{\overline{\dim}_M T_{m\mathbb{N}}^r - \log_q (r-1)}{1 - \log_q(r-1)}, \]
and 
\[ \underline{d}(G^\ast(F^r,T_{m\mathbb{N}}^r) ) = \frac{\dim T_{m\mathbb{N}}^r - \log_q (r-1)}{1 - \log_q(r-1)}. \]

\subsubsection{Sharpness of Theorem \ref{second_inverse_thm_trees}}
\label{sharpness}
Next, we modify the construction of $T_{m\mathbb{N}}^r$ to obtain for every $\eps > 0$ a tree $T_{\eps}$ with
\[
	0 < \underline{d}(G^\ast(F^r,T_{\eps}) ) = \overline{d}(G^\ast(D^{r,2}, T_\eps)) < (1+\eps) \frac{\dim T_{\eps} - \log_q (r-1)}{1 - \log_q(r-1)}
\] 
such that there exists $k \ge 1$ with $V^{r,m,k}_1$ not appearing in $T_{\eps}$. 

Let $T_{\eps} = T_E^r$, where $E = m \mathbb{N} \setminus mM \mathbb{N}$ for some positive integer $M > 1+\eps^{-1}$. 
Then 
\[ \overline{d}(E) = \frac{1}{m}\left(1-\frac{1}{M}\right)\]
and $V^{r,m,mM+1}_1$ does not appear in $T_{\eps}$. 
By the self-similarity of $T_{\eps}$ and Example \ref{Ex.2.2}, we have
\[
\dim T_\eps = \overline{\dim}_M T_{\eps} =  \frac{1}{m}\left(1-\frac{1}{M}\right)+ \log_q (r-1) \left(1 - \frac{1}{m}\left(1-\frac{1}{M}\right)\right)
\]
and hence
\[\frac{\dim T_\eps - \log_q(r-1)}{1-\log_q(r-1)} = \frac{1}{m} \left(1-\frac{1}{M} \right).\]
Since $\Delta_0^\ast(E) = m\mathbb{N}$, observe that by Remark \ref{Rmk.2.5} we have $\underline{d}(G^\ast(F^r, T_\eps)) = \overline{d}(G^\ast(D^{r,2}, T_\eps)) = m^{-1}$. 
Therefore 
\[ \underline{d}(G^\ast(F^r, T_{\eps})) = \overline{d}(G^\ast(D^{r,2},T_\eps)) = \frac{1}{1 - \frac{1}{M}} \cdot \frac{\dim T_{\eps} - \log_q (r-1)}{1 - \log_q(r-1)} < (1+ \eps) \frac{\dim T_{\eps} - \log_q (r-1)}{1 - \log_q(r-1)}.
\]

\subsection*{Dynamical setup}

Given a Markov tree $\tau$ and $v \in |\tau|$, define the Markov tree $\tau^v$ by $\tau^v(w) = \tau(vw)/\tau(v)$ for every $w \in \Lambda^\ast$. 
Using this we define a Markov process $p \colon M \to \mathcal{P}(M)$ on the space $M = \Lambda \times \mathcal{P}(\Lambda^\mathbb{N})$ by $p(a,\tau) = \sum_{i \in \Lambda} \tau(i) \delta_{(i,\tau^i)} \in \mathcal{P}(M)$.
Here $a \in \Lambda$ can be interpreted as labelling the root of $\tau \in \mathcal{P}(\Lambda^\mathbb{N})$ with information about the past under the dynamics $\tau \mapsto \tau^a$. 
Since $p$ is continuous, it induces a Markov operator $P$ on $C(M)$ (a positive contraction satisfying $P1=1$) defined by the formula $Pf(a, \tau) = \sum_{i \in \Lambda} \tau(i)f(i, \tau^i)$.
The pair $(M,p)$ is a \emph{CP-process}. 

\begin{remark}
For simplicity of notation, frequently we will denote a labelled Markov tree by its underlying Markov tree.
Similarly, we write $p_\tau = p(a,\tau)$ since the latter is independent of $a$. 
Further, a labelled Markov tree denoted by $\tau^a$ is always assumed to have label $a$. 
\end{remark}

By a \emph{distribution} we mean a Borel probability measure. 
A distribution $\nu$ on $M$ is \emph{stationary} for $(M,p)$ if $\int_M Pf \,d\nu = \int_M f\,d\nu$ for all continuous $f$. 
Note that if $\nu$ is stationary, then the above formula for $P$ extends to a well-defined operator on $L^p(M,\nu)$ for $1 \le p \le \infty$, and by Jensen's inequality this extension is a Markov operator. 

For $i \in \Lambda$, define the set $B_i = \{(a,\tau) \in M \colon a = i\}$ of Markov trees labelled by~$i$. 
The sets $B_i$ are clopen and partition $M$. 
Define also for $2 \le r \le q$ the set $A_r = \{\tau \in M \colon |\{i \colon p_\tau(B_i)>0\}| \ge r\}$ of Markov trees $\tau$ such that there are at least $r$ vertices in $|\tau|(1)$. 
Observe that $A_r$ is open and dense in $M$, and hence is not closed for $r>1$. 

Define on $M$ the \emph{information function} 
\[H(\tau) = -\sum_{i \in \Lambda} p_\tau(B_i) \log_q p_\tau(B_i) = -\sum_{i \in \Lambda} \tau(i)\log_q \tau(i),\] where by convention $0 \log_q 0 = 0$. 
The \emph{entropy} of a stationary distribution $\nu$ is then $H(\nu) = \int_M H\,d\nu$. 
Note that $0 \le H(\tau) \le \log_q \left| |\tau|(1)\right|$. 

\begin{prop}\label{a-entropy}
If $\nu$ is a stationary distibution for $(M,p)$, then
\[\nu(A_r) \ge \frac{H(\nu)-\log_q(r-1)}{1-\log_q(r-1)}.\]
\end{prop}

\begin{proof}
Using the above bounds on $H(\tau)$ and the definition of $A_r$, 
\[H(\nu) = \int_{A_r} H\,d\nu + \int_{M \setminus A_r} H\,d\nu \le \nu(A_r)+(1-\nu(A_r))\log_q(r-1).\]
Rearranging gives the proposition. 
\end{proof}

\subsection*{Endomorphic extension}

It will be necessary to work with an extension of the CP-process $(M,p)$, following \cite{FW}. 

Let $\widetilde{M} = \{\widetilde{\tau}=(\tau_i)_{i \le 0} \in M^{\mathbb{Z}^{\le 0}} \colon p_{\tau_i}(\{\tau_{i+1}\})>0 \text{ for all }i<0\}$. 
By abuse of notation we denote by $p$ the natural lift of $p \colon M \to \mathcal{P}(M)$ to a continuous function $\widetilde{M} \to \mathcal{P}(\widetilde{M})$. 
Explicitly, $p_{\tilde{\tau}} = \sum_{a \in \Lambda} \tau_0(a) \delta_{\tilde{\tau}^{a}}$, where $(\tilde{\tau}^a)_i =  \tau_{i+1}$ for $i < 0$ and $(\tilde{\tau}^a)_0 = \tau_0^a$.
We also denote by $P$ the corresponding Markov operator on $C(\widetilde{M})$. 
The pair $(\widetilde{M},p)$ is said to be an \emph{endomorphic extension} of $(M,p)$. 

A stationary distribution $\nu$ on $M$ induces a stationary distribution $\widetilde{\nu}$ on $\widetilde{M}$, and by construction $\widetilde{\nu}$ is invariant under the right shift $S \colon (\tau_i)_{i \le 0} \mapsto (\tau_{i-1})_{i \le 0}$ \cite[Definition 6.3, Remark 6.4, and Lemma 6.8]{Ho}.
The Koopman operator of $S$ therefore acts on $\mathscr{H} = L^2(\widetilde{M},\mathscr{B},\widetilde{\nu})$, where $\mathscr{B}$ is the Borel $\sigma$-algebra on $\widetilde{M}$. 
Since $p_{\widetilde{\tau}}(\{\widetilde{\omega}\})>0$ implies $S(\widetilde{\omega})=\widetilde{\tau}$, a straightforward calculation gives 
\begin{lemma}\label{lem1}
For any $f,g \in \mathscr{H}$ we have $P(f Sg) = g Pf$. \qed
\end{lemma}

Integrating with respect to $\widetilde{\nu}$ shows that $P$ and $S$ are adjoint operators on $\mathscr{H}$, and taking $f=1$ gives the formula $PS=I$. 
It follows that $S^nP^n$ is the orthogonal projection from $\mathscr{H}$ onto the closed subspace $S^n\mathscr{H} = L^2(\widetilde{M},S^{-n}\mathscr{B},\nu)$. 

If $f =Sf' \in S\mathscr{H}$ then $SPf=SPSf'=Sf'=f$, so $SP=I$ on $S\mathscr{H}$. 
Define $\mathscr{H}_\infty = \cap_{n \ge 1} S^n\mathscr{H} = L^2(\widetilde{M}, \mathscr{B}_\infty,\nu)$, where $\mathscr{B}_\infty = \cap_{n \ge 1} S^{-n}\mathscr{B}$. 
For $f \in \mathscr{H}_\infty$ we have $Sf \in \mathscr{H}_\infty$ and $Pf \in \mathscr{H}_\infty$ (using $PS=I$), giving
 
\begin{lemma}\label{lemma3}
$P$ and $S$ restrict to mutually inverse operators $\mathscr{H}_\infty \to \mathscr{H}_\infty$. \qed
\end{lemma}

Denote the orthogonal projection of $f \in \mathscr{H}$ onto $\mathscr{H}_\infty$ by $\overline{f}$. 

\begin{prop}\label{prop_4}
For $f \in \mathscr{H}$, $\|P^nf - P^n\overline{f}\|_2 \to 0$. 
\end{prop}

\begin{proof}
As $\widetilde{\nu}$ is $S$-invariant, it follows from Lemma \ref{lemma3} that
\[\|P^nf - P^n\overline{f}\|_2 = \|S^nP^nf-S^nP^n\overline{f}\|_2 =\|S^nP^nf - \overline{f}\|_2 \to 0\]
since $\|E(f \mid S^{-n}\mathscr{B} ) - E(f \mid \cap_{i \ge 1} S^{-i}\mathscr{B})\|_2 \to 0$ \cite[Theorem 5.8]{EW}. 
\end{proof}

By composing $H$ with the projection $\widetilde{M}\to M$ onto the $0$-th coordinate, the information function $H$ is defined on $\widetilde{M}$, and hence the entropy of a stationary distribution for $(\widetilde{M},p)$ is defined as for $(M,p)$. 

\section{The Furstenberg--Weiss correspondence principle}
\label{Correspondence_principle}

In \cite{FW} Furstenberg and Weiss associated to a tree of positive upper Minkowski dimension a stationary distribution for the CP-process $(M,p)$, and showed that the appearance of $D^{2,k}_n$ could be deduced from the positivity of quantities defined on the dynamical system. 
In this section we extend their construction to arbitrary configurations, and prove an analogous correspondence principle based on \cite{Fur1} for trees of positive Hausdorff dimension. 

\subsection{Construction of configuration-detecting functions}\label{phi-construction}

Given a configuration $C$ and an integer $n \ge 1$, we say that a function $f \colon M \to [0,1]$ is \emph{$C_n$-detecting} if $f(\tau)>0$ if and only if $C$ appears at the root of $|\tau|$ with parameter $n$. 
In preparation for proving correspondence principles we construct recursively several families of configuration-detecting functions. 

We first construct a set of configuration-detecting functions $\varphi_{C_n}$. 
For the simplest configuration $\{\emptyset\}$, we can take $\varphi_{\{\emptyset\}_n}=1$ for all $n \ge 1$. 
Given $I \subset \Lambda$ such that $|I|=|C(1)|$ and a bijection $\beta \colon I \to C(1)$, the positivity of $\prod_{i \in I} P(1_{B_i}P^{n-1}\varphi_{C_n^{\beta(i)}})$ at $\tau \in M$ is equivalent to the appearance of $C$ at the root of $|\tau|$ with parameter $n$ such that $\beta(i) \in C(1)$ is mapped to $iv \in |\tau|$ for some $v \in \Lambda^{n-1}$. 
Summing over all choices of $I$ and $\beta$, we define $\varphi_{C_n}$ by the recursive formula
\[\varphi_{C_n} = \sum_{\substack{I \subset \Lambda \\ |I|=|C(1)|}} \sum_{\beta \colon I \xrightarrow{\sim} C(1)} \prod_{i \in I} P(1_{B_i}P^{n-1}\varphi_{C_n^{\beta(i)}}).\]	

\begin{remark}
Alternatively we could sum over all injections $\gamma:C(1) \to \Lambda$ and define $\varphi_{C_n}$ by
\[
\varphi_{C_n} = \sum_{\gamma} \prod_{i \in C(1)} P(1_{B_{\gamma(i)}}P^{n-1}\varphi_{C_n^{i}}).
\]
\end{remark}

We also have $0 \le \varphi_{C_n} \le 1$. 
Indeed, since $\varphi_{\{\emptyset\}_n}=1$ and $P$ is positive
\[0 \le \varphi_{C_n} \le  \sum_{\substack{I \subset \Lambda \\ |I|=|C(1)|}} \sum_{\beta \colon I \xrightarrow{\sim} C(1)} \prod_{i \in I} P1_{B_i} \le \left(\sum_{i \in \Lambda}P1_{B_i}\right)^{|C(1)|} = 1.\]

Starting instead with $\phi_{D_n^{r,1}} = 1_{A_r}$ and $\phi_{C_n} = 1$ for nonbranching configurations $C$, we can adapt the above recursion to construct an alternative family of configuration-detecting functions $\phi_{C_n} \ge \varphi_{C_n}$ more suitable for computations. 
Let $C(1)' = \{v \in C(1) \colon C^v \text{ is branching}\}$. 
We define $\phi_{C_n}$ recursively by the formula
\[\phi_{C_n} = 1_{A_{|C(1)|}} \sum_{\substack{I \subset \Lambda \\ |I|=|C(1)'|}} \sum_{\beta \colon I \xrightarrow{\sim} C(1)'} \prod_{i \in I} P(1_{B_i}P^{n-1}\phi_{C_n^{\beta(i)}}).\]
Note that $\varphi_{D_n^{r,1}} \le 1_{A_r} = \phi_{D_n^{r,1}}$. 
Similarly we have $0 \le \phi_{C_n} \le 1$. 

As the $B_i$ are clopen and $P$ takes continuous functions to continuous functions, the $\varphi_{C_n}$ are continuous. 
However, the $\phi_{C_n}$ are in general not continuous since $A_r$ is not clopen for $r>1$. 

If $C$ is a configuration such that the configurations $C^v$ are all ``isomorphic'' for $v \in C(1)$, the above recursion can be simplified by omitting the sum over bijections~$\beta$. 
For integers $2 \le r \le q$ and $n \ge 1$, define (nonlinear) operators $R_{r,n}$ on $L^\infty(M)$ by
\[R_{r,n}f = \sum_{\substack{I \subset \Lambda \\ |I|=r}} \prod_{i \in I} P(1_{B_i}P^{n-1}f).\]
If $f$ detects $C^v_n$ for (all) $v \in C(1)$ and $|C(1)|=r$, then $R_{r,n}f$ detects $C_n$. 
Denote by $\phi_{C_n}'$ the $C_n$-detecting function obtained by applying a sequence of the above operators to the appropriate $1_{A_r}$, and observe that $\phi_{C_n}=c\phi_{C_n}'$ for some integer $c \ge 1$. 

\begin{eg}
	For the configuration $F^r$, we have $|C(1)|=r$ and $|C(1)'|=1$. 
There is always a unique bijection $I \to C(1)'$, so linearity of $P$ gives
\[ \phi_{F^r_n} = 1_{A_r} \sum_{i \in \Lambda} P(1_{B_i}P^{n-1}1_{A_r}) = 1_{A_r}P^n1_{A_r}\]
since $1=\sum_{i \in \Lambda}1_{B_i}$. 

	If $C(1)=C(1)'$, the factor $1_{A_{|C(1)|}}$ is redundant in the definition of $\phi_{C_n}$ as the function in the sum is already supported on a subset of $A_{|C(1)|}$. 
For example,
\[\phi_{D^{r,2}_n} =  \sum_{\substack{I \subset [q] \\ |I|=r}} \sum_{\beta \colon I \xrightarrow{\sim} \Lambda_r} \prod_{i \in I} P(1_{B_i}P^{n-1}1_{A_r}) = r! \sum_{\substack{I \subset [q] \\ |I|=r}} \prod_{i \in I} P(1_{B_i}P^{n-1}1_{A_r} ) = r!\phi_{D^{r,2}_n}'.\]
\end{eg}

The following lemma is used in the proofs of the correspondence principles to account for the lack of continuity of $\phi_{C_n}$. 

\begin{lemma}\label{approximation}
If $(\nu_k)_{k \ge 1}$ is a sequence of distributions on $M$ converging to $\nu$ in the weak-${}^\ast$ topology, then for every configuration $C$ and integer $n \ge 1$
\[\limsup_{k \to \infty} \int_M\phi_{C_n}\,d\nu_k \ge \int_M \phi_{C_n}\,d\nu.\]
\end{lemma}

\begin{proof}
Define for $\delta \in [0,1]$ open sets $A_r^\delta = \{\tau \in M \colon |\{i \colon p_\tau(B_i)>\delta\}|\ge r\} \subset A_r$, and let $\phi_{C_n}^\delta$ be the function obtained by replacing $1_{A_r}$ with $1_{A_r^\delta}$ in the recursive construction of $\phi_{C_n}$. 
Observe that $\delta \le \delta'$ implies $\phi_{C_n}^\delta \ge \phi_{C_n}^{\delta'}$ by the positivity of $P$. 
Then the monotone function $\alpha \colon \delta \mapsto \int_M \phi_{C_n}^\delta\,d\nu$ has countably many discontinuities, so we can choose a sequence $\delta_j \to 0$ such that $\alpha$ is continuous at $\delta_j$ for all $j$. 

We claim $\int_M \phi_{C_n}^\delta\,d\nu_k \to \int_M \phi_{C_n}^\delta\,d\nu = \alpha(\delta)$ if $\alpha$ is continuous at $\delta$. 
If $\delta < \delta'$, the closed sets $(A_r^\delta)^c$ and $\overline{A_r^{\delta'}} = \{\tau \in M \colon |\{i \colon p_\tau(B_i) \ge \delta'\}| \ge r\}$ are disjoint since $\overline{A_r^{\delta'}} \subset A_r^\delta$. 
By Urysohn's lemma there are continuous functions $h_r$ such that $1_{A_r^{\delta'}} \le h_r \le 1_{A_r^\delta}$. 
Defining $h_{C_n}$ to be the function obtained by repeating the construction of $\phi_{C_n}$ with $h_r$ in place of $1_{A_r}$, it follows that $\phi_{C_n}^{\delta'} \le h_{C_n} \le \phi_{C_n}^\delta$. 
Since $h_{C_n}$ is continuous, 
\begin{align*}
\liminf_{k \to \infty} \int_M \phi_{C_n}^\delta \,d\nu_k & \ge \liminf_{k \to \infty} \int_M h_{C_n}\,d\nu_k = \int_M h_{C_n}\,d\nu \ge \int_M \phi_{C_n}^{\delta'}\,d\nu = \alpha(\delta').
\end{align*}
Continuity of $\alpha$ at $\delta$ implies $\liminf_{k \to \infty} \int_M \phi_{C_n}^\delta \,d\nu_k \ge \alpha(\delta)$, and a similar argument with $\delta'<\delta$ proves the claim. 
Hence
\begin{align*}
\limsup_{k \to \infty} \int_M \phi_{C_n} \,d\nu_k &\ge \lim_{k \to \infty} \int_M \phi_{C_n}^{\delta_j}\,d\nu_k = \int_M \phi_{C_n}^{\delta_j}\,d\nu \xrightarrow{j \to \infty} \int_M \phi_{C_n}\,d\nu
\end{align*}
by the monotone convergence theorem.
\end{proof}

\subsection{Correspondence principle for upper density}

\begin{thm}\label{FW-correspondence}
For every tree $T$ with $\overline{\dim}_M T >0$, the CP-process $(M,p)$ has a stationary distribution $\mu$ such that $H(\mu)=\overline{\dim}_M T$,
\begin{equation}
\label{a-mdim}
\mu(A_r) \ge \frac{\overline{\dim}_MT - \log_q(r-1)}{1-\log_q(r-1)},
\end{equation}
and for every configuration $C$ and every integer $n \ge 1$
\begin{equation}
\label{ineq_1}
\overline{d}(C_n) \ge \int_M \phi_{C_n} \,d \mu.
\end{equation}
\end{thm}

\begin{proof}

Let $(L_k)_{k\ge 1}$ be an increasing sequence such that 
\[\overline{\dim}_M T = \lim_{k \to \infty} \frac{\log_q |T(L_k+1)|}{L_k+1}.\]
Fix an arbitrary label $a \in \Lambda$, and for each $k \ge 1$ let $\pi_k$ be any Markov tree labelled by $a$ such that $\pi_k(v) = |T(L_k)|^{-1}$ for all $v \in T(L_k)$ (note that this condition determines $\pi_k$ on vertices of level at most $L_k$). 
Then any weak-${}^\ast$ limit of the distributions
\[\mu_k = \frac{1}{L_k+1} \sum_{i=0}^{L_k} P^i\delta_{\pi_k} = \frac{1}{L_k+1} \sum_{l(v)\le L_k} \pi_k(v) \delta_{\pi^v_k}\]
is stationary, and we choose $\mu$ to be such a limit. 

Since $H(x)$ is continuous and $\pi_k(v) = \sum_{a \in \Lambda} \pi_k(va)$, 
\begin{align}\label{entropy_boundary}
H(\mu) &= \lim_{k \to \infty} \int_M H \,d\mu_{k} \nonumber \\ \nonumber
&= -\lim_{k \to \infty} \frac{1}{L_k+1} \sum_{l(v) \le L_k} \pi_k(v) \sum_{a \in \Lambda} \frac{\pi_k(va)}{\pi_k(v)} \log_q \frac{\pi_k(va)}{\pi_k(v)} \\ \nonumber
&= -\lim_{k \to \infty} \frac{1}{L_k+1} \sum_{l(v) \le L_k} \sum_{a \in \Lambda} \pi_k(va)\log_q \pi_k(va) - \pi_k(va)\log_q \pi_k(v) \\
&= -\lim_{k \to \infty} \frac{1}{L_k+1} \sum_{l(v)=L_k} \sum_{a \in \Lambda} \pi_k(va)\log_q \pi_k(va).
\end{align}
Recall that for every $v \in |\pi_k|$ we have the bounds
\begin{equation}\label{entropy_bounds}
-\pi_k(v)\log_q\pi_k(v) \le -\sum_{a \in \Lambda} \pi_k(va)\log_q\pi_k(va) \le -\pi_k(v) \log_q \frac{\pi_k(v)}{q}.
\end{equation}
Since $-\sum_{l(v)=L_k} \pi_k(v) \log_q \pi_k(v) = \log_q|T(L_k)|$ by definition of $\pi_k$, summing the inequalities (\ref{entropy_bounds}) over $v \in L_k$ and noticing $\sum_{l(v)=L_k}\pi_k(v)=1$ gives
\[H(\mu)=\int_X H\,d\mu = \lim_{k \to \infty} \frac{\log_q|T(L_k)|}{L_k+1} = \lim_{k \to \infty} \frac{\log_q|T(L_k+1)|}{L_k+1} = \overline{\dim}_MT.\]
where the third equality follows from the bounds 
\[q^{-1}|T(L_k+1)| \le |T(L_k)| \le |T(L_k+1)|.\] 
Proposition \ref{a-entropy} immediately gives the inequality (\ref{a-mdim}). 

To prove the inequality (\ref{ineq_1}), applying a change of summation variable and using the definitions of $\pi_k$ and $\phi_{C_n}$ gives
\begin{align*}
\overline{d}(C_n) &\ge \limsup_{k \to \infty} \frac{1}{|T(L_k)|}\sum_{v \in T(L_k)} \frac{|C_n \cap \{w \in T \colon w \le v\}|}{L_k+1} \\
&= \limsup_{k \to \infty} \frac{1}{L_k+1} \sum_{l(w) \le L_k} \pi_k(w) 1_{C_n}(w) \\
&\ge \limsup_{k \to \infty} \frac{1}{L_k+1} \sum_{l(w) \le L_k} \pi_k(w) \phi_{C_n}(\pi_k^w) \\
&= \limsup_{k \to \infty} \int_M \phi_{C_n}\,d\mu_k. 
\end{align*}
The conclusion follows from Lemma \ref{approximation}. 
\end{proof}

\subsection{Correspondence principle for upper Banach density}

\begin{thm}\label{ergodic-correspondence}
If $\dim T >0$, for every $\epsilon >0$ there exists an ergodic\footnote{Ergodicity here means an extremal point in the compact convex subset of all stationary distributions.} stationary distribution $\eta = \eta_\epsilon$ for the CP-process $(M,p)$ such that $H(\eta) \ge \dim T-\epsilon$,
\begin{equation}
\label{a-hdim}
\eta(A_r) \ge \frac{\dim T-\epsilon-\log_q(r-1)}{1-\log_q(r-1)}, 
\end{equation}
and for every configuration $C$ and integer $n \ge 1$
\begin{equation}\label{ergodic_correspondence_ineq}
d^\ast(C_n) \ge \int_M \phi_{C_n}\,d\eta. 
\end{equation}
\end{thm}

\begin{proof}
For any $\epsilon>0$, by Frostman's lemma (see \cite[Theorem 8.8]{Ma} and \cite[Theorem 3.12]{Ho}) there exists $\theta \in M$ such that $|\theta| \subset T$ and $\dim \theta \ge \dim T - \epsilon$. 
Let $(M_k)_{k\ge 1}$ be an increasing sequence such that the distributions 
\[\eta'_k = \frac{1}{M_k+1} \sum_{i=0}^{M_k} P^i\delta_\theta = \frac{1}{M_k+1}\sum_{l(v) \le M_k} \theta(v)\delta_{\theta^v}.\]
converge to a distribution $\eta'$ in the weak-${}^\ast$ topology. 

\begin{lemma} \cite[Lemma 4]{Fur1}
$H(\eta') \ge \dim \theta$. 
\end{lemma}

\begin{proof}
As in the calculation (\ref{entropy_boundary}) we have
\[ H(\eta') = -\lim_{k \to \infty} \frac{1}{M_k+1}\sum_{l(v)=M_k+1} \theta(v)\log_q\theta(v) = \lim_{k \to \infty}\frac{-\sum_{v \in \Pi_k}\theta(v)\log_q \theta(v)}{\sum_{v \in \Pi_k} l(v)\theta(v)} \ge \dim \theta \]
where $\Pi_k$ is the section $|\theta|(M_k+1) = \{v \in |\theta| \colon l(v)=M_k+1\}$. 
\end{proof}

The support of $\eta'$ is contained in the compact set $D(\theta) = \overline{\{\theta^v \colon v \in |\theta|\}}$, and by Choquet's theorem \cite[Chapter 3]{Phelps} there exists an ergodic distribution $\eta$ supported on $D(\theta)$ such that $H(\eta) \ge H(\eta') \ge \dim \theta \ge \dim T - \epsilon$. 
The inequality (\ref{a-hdim}) immediately follows from Proposition \ref{a-entropy}. 

Since $\eta$ is ergodic, the mean ergodic theorem for contractions \cite[Theorem 8.6]{EFHN} implies 
\[\frac{1}{N+1} \sum_{i=0}^N P^if \to \int_M f\,d\eta\]
in $L^2(M,\eta)$ for $f \in L^2(M,\eta)$. 
By diagonalisation there exists an increasing sequence $(N_k)_{k\ge 1}$ and $\tau \in D(\theta)$ such that 
\begin{equation}\label{generic}
\frac{1}{N_k+1} \sum_{i=0}^{N_k} P^if(\tau) \to \int_M f\,d\eta
\end{equation}
for all $f$ in a countable set of continuous functions. 
Taking this set to be dense in $C(M)$ under the uniform norm, the limit (\ref{generic}) holds for all continuous functions. 
Letting $v_k \in |\theta|$ be a sequence of vertices such that $\theta^{v_k} \to \tau$, and passing to a subsequence of $(v_k)$ if necessary, it follows that the sequence of measures $\eta_k$ defined by
\[\eta_k = \frac{1}{N_k+1}\sum_{i=0}^{N_k}P^i\delta_{\theta^{v_k}} = \frac{1}{N_k+1}\sum_{l(w)\le N_k}\theta^{v_k}(w)\delta_{\theta^{v_kw}}\]
converges weakly to $\eta$. 
For $\epsilon <\dim T$ it follows that
\begin{align*}
d^\ast(C_n) &\ge \limsup_{k \to \infty} \frac{1}{N_k+1} \sum_{l(w) \le N_k} \theta^{v_k}(w)1_{C_n}(v_k w) \\
&\ge \limsup_{k \to \infty} \frac{1}{N_k+1} \sum_{l(w) \le N_k} \theta^{v_k}(w)\phi_{C_n}(\theta^{v_kw}) \\
&= \limsup_{k \to \infty} \int_M \phi_{C_n} \,d\eta_k
\end{align*}
and the inequality (\ref{ergodic_correspondence_ineq}) follows from Lemma \ref{approximation}. 
\end{proof}

\begin{remark}
Composing the projection $(\tau_i)_{i \le 0} \mapsto \tau_0$ with a $C_n$-detecting function gives a map $\widetilde{M} \to [0,1]$ which is positive at $(\tau_i)_{i \le 0}$ if and only if $C$ appears at the root of $|\tau_0|$ with parameter~$n$. 
The recursive constructions of configuration-detecting functions in Subsection \ref{phi-construction} can be used to construct their lifts using the abuses of notation $B_i = \{\widetilde{\tau} \in M \colon \tau_0 \in B_i\}$ and $A_r = \{\widetilde{\tau} \in \widetilde{M} \colon |\{i \colon p_{\widetilde{\tau}}(B_i)>0\}| \ge r\}$. 
Observe that the inequalities (\ref{a-mdim}), (\ref{ineq_1}), (\ref{a-hdim}), and (\ref{ergodic_correspondence_ineq}) are still valid when the distributions $\mu$ and $\eta_\epsilon$ and the configuration detecting functions $\phi_{C_n}$ are replaced with their lifts on $\widetilde{M}$. 
In the remainder of the paper we work only with $(\widetilde{M},p)$ and use Theorems \ref{FW-correspondence} and \ref{ergodic-correspondence} for the endomorphic extension without comment. 
\end{remark}

\section{Proof of direct theorems}
\label{Direct_Theorems}

Using the correspondence principles of Section \ref{Correspondence_principle}, we bound from below the densities of the sets of popular differences for trees. 
We first prove such a result for the generic parameters of the configuration $F^r$, since the proof is relatively simple but contains the main ideas. 

\begin{thm}\label{Theorem4.1}
Let $T$ be a tree. 
For $2 \le r \le q$ we have
\[\underline{d}(\overline{G}(F^r)) \ge \frac{\overline{\dim}_MT - \log_q(r-1)}{1-\log_q(r-1)} \mbox{  and  }
\underline{d}(G^*(F^r)) \ge \frac{\dim T - \log_q(r-1)}{1-\log_q(r-1)}.\]
\end{thm}
\begin{proof}
Since $P$ and $S$ are adjoint, Theorem \ref{FW-correspondence} gives 
\[\overline{d}(F_n^r) \ge \int_{\widetilde{M}} \phi_{F_n^r}\,d\widetilde{\mu} = \int_{\widetilde{M}} 1_{A_r} P^n 1_{A_r} \,d \widetilde{\mu} = \int_{\widetilde{M}} 1_{A_r}S^n1_{A_r}\,d\widetilde{\mu} = \widetilde{\mu}(A_r \cap S^{-n}A_r).\]
Hence $\overline{G}(F^r) \supset \mathcal{R} = \{n \in \mathbb{N} \colon \widetilde{\mu}(A_r \cap S^{-n}A_r)>0\}$, so $\underline{d}(\overline{G}(F^r)) \ge \underline{d}(\mathcal{R})$. 
By the mean ergodic theorem
\[\underline{d}(\mathcal{R}) = \liminf_{N \to \infty} \frac{1}{N+1} \sum_{n=0}^N 1_{\mathcal{R}}(n) \ge \liminf_{N \to \infty}\frac{1}{N+1} \sum_{n=0}^N \frac{\widetilde{\mu}(A_r \cap S^{-n}A_r)}{\widetilde{\mu}(A_r)} \ge \widetilde{\mu}(A_r),\]
and the theorem follows from inequality~(\ref{a-mdim}) of Theorem \ref{FW-correspondence}. 

Using Theorem~\ref{ergodic-correspondence} in place of Theorem~\ref{FW-correspondence} in the above argument, we obtain the second inequality after taking $\epsilon \to 0$. 
\end{proof}

Theorem \ref{Theorem4.1} is also immediate from the corresponding result for $D^{r,2}$, which we prove now.

\begin{thm}\label{direct-D-upper}
Let $T$ be a tree. 
For $2 \le r \le q$ we have
\[\underline{d}(\overline{G}(D^{r,2})) \ge \frac{\overline{\dim}_MT - \log_q(r-1)}{1-\log_q(r-1)}
\mbox{  and  } 
\underline{d}(G^*(D^{r,2})) \ge \frac{\dim T - \log_q(r-1)}{1-\log_q(r-1)}.\]
\end{thm}

\begin{proof}

We start with the proof of the first inequality. The idea is to show that $\overline{G}(D^{r,2})$ essentially contains the set of return times of $A_r$, the density of which we can bound from below by the mean ergodic theorem. 
First observe that by Proposition \ref{prop_4} 
\begin{multline*}
\left| \int_{\widetilde{M}} \phi_{D_n^{r,2}}\,d\widetilde{\mu}-r!\int_{\widetilde{M}} \sum_{\substack{I \subset [q] \\ |I|=r}} \prod_{i \in I} P(1_{B_i}P^{n-1}\overline{1_{A_r}} ) \,d \widetilde{\mu} \right| \\
=\left|r!\int_{\widetilde{M}} \sum_{\substack{I \subset [q] \\ |I|=r}} \prod_{i \in I} P(1_{B_i}P^{n-1}1_{A_r} ) \,d \widetilde{\mu} -r!\int_{\widetilde{M}} \sum_{\substack{I \subset [q] \\ |I|=r}} \prod_{i \in I} P(1_{B_i}P^{n-1}\overline{1_{A_r}} ) \,d \widetilde{\mu} \right|  \xrightarrow{n\to\infty} 0.
\end{multline*}
Since $\overline{1_{A_r}} \in \mathscr{H}^\infty$, by Lemma \ref{lemma3} $P^{n-1}\overline{1_{A_r}} = SP^n\overline{1_{A_r}}$. 
Then by Lemma~\ref{lem1} and orthogonality
\[r!\int_{\widetilde{M}} \sum_{\substack{I \subset [q] \\ |I|=r}} \prod_{i \in I} P(1_{B_i}P^{n-1}\overline{1_{A_r}} ) \,d \widetilde{\mu} = \int_{\widetilde{M}} \overline{\varphi_{D^{r,1}_n}}(P^n\overline{1_{A_r}})^r\,d\widetilde{\mu},\]
recalling $\varphi_{D^{r,1}_n} = r!\sum_{\substack{I \subset [q] \\ |I|=r}} \prod_{i \in I} P1_{B_i}$. 
Define $Z_\rho = \{\widetilde{\tau} \in \widetilde{M} \colon \overline{\varphi_{D^{r,1}_n}}(\widetilde{\tau})\ge \rho\}$, and observe it is well-defined up to a $\widetilde{\mu}$-null set. 
Since $0 \le \varphi_{D^{r,1}_n} \le 1_{A_r} \le 1$ and $0 \le \rho 1_{Z_\rho} \le \overline{\varphi_{D^{r,1}_n}} \le 1$, the positivity of both $P$ and conditional expection imply
\[\int_{\widetilde{M}} \overline{\varphi_{D^{r,1}_n}}(P^n\overline{1_{A_r}})^r\,d\widetilde{\mu} \ge \int_{\widetilde{M}} \overline{\varphi_{D^{r,1}_n}}(P^n\overline{\varphi_{D^{r,1}_n}})^r\,d\widetilde{\mu} \ge \rho^{r+1}\int_{\widetilde{M}} 1_{Z_\rho} (P^n 1_{Z_\rho})^r\,d\widetilde{\mu}.\]
By Jensen's inequality and the adjointness of $P$ and $S$
\[\int_{\widetilde{M}} 1_{Z_\rho} (P^n 1_{Z_\rho})^r\,d\widetilde{\mu} \ge \left(\int_{\widetilde{M}} 1_{Z_\rho}P^n1_{Z_\rho} \,d\widetilde{\mu}\right)^r = \widetilde{\mu}(Z_\rho \cap S^{-n}Z_\rho)^r.\] 
Combining the above with the correspondence principle Theorem \ref{FW-correspondence}, it follows that $\overline{G}(D^{r,2})$ contains a cofinite subset of 
\[\mathcal{R}^\delta(Z_\rho) = \{n \in \mathbb{N} \colon \widetilde{\mu}(Z_\rho \cap S^{-n}Z_\rho) > \delta\widetilde{\mu}(Z_\rho)^2\}\]
for all $\delta, \rho >0$ (since $S$ is $\widetilde{\mu}$-preserving). 
Therefore
\begin{align*}
\underline{d}(\overline{G}(D^{r,2})) \ge \underline{d}(\mathcal{R}^\delta(Z_\rho)) &\ge \liminf_{N \to \infty} \frac{1}{N+1} \sum_{\substack{n \le N \\ n \in \mathcal{R}^\delta(Z_\rho)}} \frac{\widetilde{\mu}(Z_\rho \cap S^{-n}Z_\rho)}{\widetilde{\mu}(Z_\rho)} \\
&\ge \left( \liminf_{N \to \infty} \frac{1}{N+1} \sum_{n=0}^N \frac{\widetilde{\mu}(Z_\rho \cap S^{-n}Z_\rho)}{\widetilde{\mu}(Z_\rho)} \right) - \delta\widetilde{\mu}(Z_\rho) \\
&\ge (1-\delta)\widetilde{\mu}(Z_\rho) \xrightarrow{\delta \to 0} \widetilde{\mu}(Z_\rho) \xrightarrow{\rho \to 0} \widetilde{\mu}(Z),
\end{align*}
where the last inequality follows from the mean ergodic theorem and 
\[Z = \{\widetilde{\tau} \in \widetilde{M} \colon \overline{\varphi_{D^{r,1}_n}}(\widetilde{\tau}) >0\}.\] 
Since $1_{Z^c} \in L^\infty(\widetilde{M},\mathscr{B}_\infty)$, properties of the conditional expectation give 
\[0=\int_{\widetilde{M}} 1_{Z^c} \overline{\varphi_{D^{r,1}_n}}\,d\widetilde{\mu} = \int_{\widetilde{M}} 1_{Z^c}\varphi_{D^{r,1}_n}\,d\widetilde{\mu}.\] 
Hence $Z \supset \{\widetilde{\tau} \in \widetilde{M} \colon \varphi_{D^{r,1}_n}(\widetilde{\tau})>0\} = A_r$ up to a $\widetilde{\mu}$-null set, so $\underline{d}(\overline{G}(D^{r,2})) \ge \widetilde{\mu}(A_r)$. 
The theorem then follows from inequality~(\ref{a-mdim}) of Theorem \ref{FW-correspondence}. 

Using Theorem~\ref{ergodic-correspondence} in place of Theorem~\ref{FW-correspondence} in the above proofs, we obtain the second inequality after taking $\epsilon \to 0$. 
\end{proof}

\section{Inverse theorems for return times}
\label{Inverse_theorems_measure}

Let $(X,\mathscr{B},\nu,S)$ be a measure-preserving system, and let $A$ be a measurable set with $\nu(A)>0$. 
If $\mathcal{R} = \{n \in \mathbb{N} \colon \nu(A \cap S^{-n}A) >0\}$ is the set of return times of $A$, then by the mean ergodic theorem $\underline{d}(\mathcal{R}) \ge \nu(A)$. 

\begin{thm} \label{first_inverse_thm_measure}
If $\overline{d}(\mathcal{R}) = \nu(A) > 0$, then there exists an integer $m \ge 1$ such that up to $\nu$-null sets
\[X = \bigsqcup_{i=0}^{m-1} S^{-i} A. \]
\end{thm} 

\begin{proof}

Define $\mathcal{R}_\gamma = \{n \in \mathbb{N} \colon \nu(A \cap S^{-n}A) \ge (1-\gamma)\nu(A) \}$. 

\begin{lemma}\label{lemma4}
	If $n \in \mathcal{R}_\gamma$ and $n' \in \mathcal{R}_{\gamma'}$, then $n + n' \in \mathcal{R}_{\gamma+\gamma'}$. 
\end{lemma}

\begin{proof}
If $B \subset A$ then $\nu(A \cap S^{-n}B) \ge \nu(B)-\gamma\nu(A)$. 
For $B = A \cap S^{-n'}A$ we have 
\[\nu(A \cap S^{-(n+n')}A)  \ge \nu(A \cap S^{-n}(A \cap S^{-n'}A)) \ge \nu(A \cap S^{-n'}A)-\gamma\nu(A) \ge (1-\gamma-\gamma')\nu(A),\]
so $n+n'\in \mathcal{R}_{\gamma+\gamma'}$. 
\end{proof}

\begin{lemma}\label{lemma5}
If $0 < \gamma < \frac{1}{2}$, then $d(\mathcal{R}_\gamma+\mathcal{R}_\gamma)=d(\mathcal{R}_\gamma)=d(\mathcal{R})$. 
\end{lemma}

\begin{proof}

Let $(N_k)_{k\ge 1}$ be an increasing sequence such that 
\[d_{(N_k)}(\mathcal{R}_\gamma) = \lim_{k \to \infty} \frac{|\mathcal{R}_\gamma \cap \{0,\ldots,N_k\}|}{N_k+1}\]
 exists. 
By the mean ergodic theorem 
\begin{align*}
\nu(A) &\le \lim_{k \to \infty} \frac{1}{N_k+1} \sum_{n=0}^{N_k} \frac{\nu(A \cap S^{-n}A)}{\nu(A)} = \lim_{k \to \infty} \frac{1}{N_k+1} \sum_{n=0}^{N_k} 1_\mathcal{R}(n) \frac{\nu(A \cap S^{-n}A)}{\nu(A)} \\
&\le \lim_{k \to \infty} \frac{1}{N_k+1} \left( \sum_{n \in \mathcal{R}_\gamma, n \le N_k} 1_{\mathcal{R}}(n) + \sum_{n \notin \mathcal{R}_\gamma,n\le N_k} 1_\mathcal{R}(n)(1-\gamma)\right) \\
&= d_{(N_k)}(\mathcal{R}_\gamma) + (1-\gamma)(d(\mathcal{R})-d_{(N_k)}(\mathcal{R}_\gamma)). 
\end{align*}
Rearranging and using the assumption $d(\mathcal{R})=\nu(A)$, it follows that
\[\nu(A) \le d_{(N_k)}(\mathcal{R}_\gamma) \le d(\mathcal{R})=\nu(A).\]
Hence $d_{(N_k)}(\mathcal{R}_\gamma)=d(\mathcal{R})$. 
By Lemma \ref{lemma4} $\mathcal{R}_\gamma+\mathcal{R}_\gamma \subset \mathcal{R}_{2\gamma} \subset \mathcal{R}$, so
	\[d(\mathcal{R})=d_{(N_k)}(\mathcal{R}_\gamma) = \underline{d}_{(N_k)}(\mathcal{R}_\gamma) \le \underline{d}_{(N_k)}(\mathcal{R}_\gamma+\mathcal{R}_\gamma) \le \overline{d}_{(N_k)}(\mathcal{R}_\gamma+\mathcal{R}_\gamma) \le d(\mathcal{R}).\]
Hence $d_{(N_k)}(\mathcal{R}_\gamma+\mathcal{R}_\gamma)$ exists and equals $d(\mathcal{R})$.
Since $(N_k)_{k \ge 1}$ was arbitrary, the conclusion follows. 
\end{proof}

For $0 < \gamma < \frac{1}{2}$, Lemma \ref{lemma5} and Kneser's theorem \cite{Kn} (see also \cite[Theorem 1.1]{Bi}) therefore imply the existence of $m \ge 1$ and $K \subset \{0,\ldots,m-1\}$ such that 
\begin{itemize}
\item 
$\mathcal{R}_\gamma \subset K + m \mathbb{N}$,
\item 
$|K+K| = 2 |K| -1$, where the operation on the left hand side is in $ \mathbb{Z} /m \mathbb{Z}$, and
\item 
$\mathcal{R}_\gamma + \mathcal{R}_\gamma \subset K+K +m\bN$ with $ |(K+K +m\bN) \setminus (\mathcal{R}_\gamma + \mathcal{R}_\gamma)| < \infty$.
\end{itemize}
It follows that $K = \{0\}$, so $\mathcal{R}_\gamma \subset m\mathbb{N}$ and $d(\mathcal{R})=d(\mathcal{R}_\gamma)=d(\mathcal{R}_\gamma+\mathcal{R}_\gamma)=m^{-1}$. 
Further, for all $n \in \mathcal{R}$ there exists $\gamma>0$ small enough such that $n+\mathcal{R}_\gamma \subset \mathcal{R}$ by Lemma \ref{lemma4}. 
Since in addition $\mathcal{R}_\gamma \subset \mathcal{R}$ and $d(\mathcal{R})=d(\mathcal{R}_\gamma)$, it follows that $n \in m\mathbb{N}$. 
Then the $m$ sets $S^{-i}A$, $0\le i \le m-1$ are disjoint (up to $\nu$-null sets) and each of measure $m^{-1}$. 
\end{proof}

\begin{thm} \label{second_inverse_thm_measure}
If $(X,\mathscr{B},\nu,S)$ is ergodic and 
\[0< \underline{d}(\mathcal{R}) < \frac{3}{2} \nu(A),\]
then there exists an integer $m \ge 1$ such that $\mathcal{R} = m \bN$ and 
$X = \bigsqcup_{i=0}^{m-1} S^{-i} \left( \bigcup_{j=0}^{\infty} S^{-jm}A \right)$ up to $\nu$-null sets. 
\end{thm}

\begin{proof}
By \cite[Section 1.5]{BFS} all the theorems in \cite{BFS} hold for ergodic $\mathbb{N}$-actions, so \cite[Theorem 1.3]{BFS} gives the existence of $m \ge 1$ such that $\mathcal{R} = m \mathbb{N}$.
Therefore, the sets 
\[\bigcup_{j=0}^{\infty} S^{-jm}A , S^{-1}\left(\bigcup_{j=0}^{\infty} S^{-jm}A\right) ,\ldots,S^{-(m-1)}\left(\bigcup_{j=0}^{\infty} S^{-jm}A \right)\]
 are mutually disjoint up to $\nu$-null sets, and ergodicity implies that they partition $X$. 
\end{proof} 

\section{Inverse theorems for trees}
\label{inverse_theorems_trees}

In this section we prove inverse results for Theorems \ref{Theorem4.1} and \ref{direct-D-upper} (Theorems \ref{first_inverse_thm_trees}, \ref{second_inverse_thm_trees}, and \ref{third_inverse_thm_trees}) using the results of the previous section. 

\begin{thm}\label{first_inverse_thm_trees}
If $T$ is a tree and $2 \le r \le q$ with
\[\overline{d}(\overline{G}(F^r)) =\frac{\overline{\dim}_MT - \log_q(r-1)}{1-\log_q(r-1)}>0,\]
then $\overline{\dim}_M T =  m^{-1}(1-\log_q(r-1))+\log_q(r-1)$ for some positive integer $m$. 
Moreover, $\overline{d}(V_1^{r,m,mk}) >0$ for every $k \ge 1$. 
\end{thm}

\begin{proof}

Let $\mathcal{R} = \{n \in \mathbb{N} \colon \widetilde{\mu}(A_r \cap S^{-n}A_r)>0\}$. 
Combining the proof of Theorem \ref{Theorem4.1}, Theorem \ref{first_inverse_thm_measure}, and the hypothesis we obtain
\[\overline{d}(\overline{G}(F^r)) \ge \overline{d}(\mathcal{R}) \ge \underline{d}(\mathcal{R}) \ge \widetilde{\mu}(A_r) \ge \frac{\overline{\dim}_MT - \log_q(r-1)}{1-\log_q(r-1)}=\overline{d}(\overline{G}(F^r)).\]
Therefore $\widetilde{\mu}(A_r) = \overline{d}(\overline{G}(F^r)) = \overline{d}(\mathcal{R})=\underline{d}(\mathcal{R})$ is positive; by Theorem \ref{first_inverse_thm_measure} it equals $m^{-1}$ for some positive integer $m$, and $\widetilde{M} =\bigsqcup_{i=0}^{m-1} S^{-i} A_r$ up to $\widetilde{\mu}$-null sets. 

The above also shows that equality holds in Proposition \ref{a-entropy} for $\widetilde{\mu}$, whence $\int_{A_r} H\,d\widetilde{\mu} = \widetilde{\mu}(A_r)$ and $\int_{A_r^c}H\,d\widetilde{\mu} = (1-\widetilde{\mu}(A_r))\log_q(r-1)$.
The bounds on $H$ then imply $\widetilde{\mu}$-almost everywhere equalities 
\begin{equation}\label{eq1}
\prod_{i \in \Lambda} P1_{B_i}=c_11_{A_r}=c_11_{A_q}
\end{equation}
\begin{equation}\label{eq2}
1_{A_{r-1}}\sum_{\substack{I \subset \Lambda \\ |I|=r-1}} \prod_{i \in I} P1_{B_i} = 1_{A_r^c}\sum_{\substack{I \subset \Lambda \\ |I|=r-1}} \prod_{i \in I} P1_{B_i} = c_21_{A_r^c}, 
\end{equation}
where $c_1 = q^{-q}$ and $c_2 = (r-1)^{1-r}$. 

Recall from Section \ref{phi-construction} the operators $R_{r-1,1}$ and $R_{q,1}$ on $L^\infty(\widetilde{M},\widetilde{\mu})$, which for simplicity we denote by $R_1$ and $R_2$: 
\[R_1f = \sum_{\substack{I \subset \Lambda \\ |I|=r-1}} \prod_{i \in I} P(1_{B_i}f); \qquad R_2f = \prod_{i \in \Lambda} P(1_{B_i}f).\]

Using the facts determined above we compute $\phi_{V_1^{r,m,mk}}' =(R_2R_1^{m-1})^k1_{A_q}$. 
In the following, equalities are only up to $\widetilde{\mu}$-null sets. 
We compute first the case $k=1$. 
Note that $A_q = S^{-m}A_q$ and $1_{S^{-i}A_q} = S1_{S^{-i+1}A_q}$ for $i \ge 1$. 
By Lemma \ref{lem1} 
\begin{align*}
R_11_{A_q} = \sum_{\substack{I \subset \Lambda \\ |I|=r-1}} \prod_{i \in I} P(1_{B_i}1_{A_q}) &= \sum_{\substack{I \subset \Lambda \\ |I|=r-1}}\prod_{i \in I} P(1_{B_i}S1_{S^{-m+1}A_q}) \\
&= 1_{S^{-m+1}A_q}\sum_{\substack{I \subset \Lambda \\ |I|=r-1}} \prod_{i \in I} P1_{B_i} \\
&= c_21_{S^{-m+1}A_q}
\end{align*}
where the last equality follows from (\ref{eq2}) and the fact $S^{-m+1}A_q =S^{-m+1}A_r \subset A_r^c$. 
Since $R_1$ is homogeneous of degree $r-1$, repeating this calculation gives
\[R_1^{m-1}1_{A_q}=c_2^{\sum_{j=0}^{m-2}(r-1)^j}1_{S^{-1}A_q}\]
and hence
\begin{align*}
\phi_{V_1^{r,m,m}}' &= R_2\left(c_2^{\sum_{j=0}^{m-2}(r-1)^j}1_{S^{-1}A_q}\right) \\
&= c_2^{q\sum_{j=0}^{m-2}(r-1)^j}\prod_{i \in \Lambda}P(1_{B_i}S1_{A_q}) \\
&= c_2^{q\sum_{j=0}^{m-2}(r-1)^j}1_{A_q}\prod_{i \in \Lambda} P1_{B_i} \\
&= c_1c_2^{q\sum_{j=0}^{m-2}(r-1)^j}1_{A_q}. 
\end{align*}
Letting $d_1 = c_1c_2^{q\sum_{j=0}^{m-2}(r-1)^j}$ and defining inductively $d_k = d_{k-1}^{q(r-1)^{m-1}}d_1$, it follows that $\phi_{V_1^{r,m,mk}}' = d_k1_{A_q}$ $\widetilde{\mu}$-almost everywhere. 
Then by the correspondence principle Theorem \ref{FW-correspondence}
\[\overline{d}(V_1^{r,m,mk}) \ge \int_{\widetilde{M}} \phi_{V_1^{r,m,mk}}\,d\widetilde{\mu} \ge \int_{\widetilde{M}} \phi_{V_1^{r,m,mk}}'\,d\widetilde{\mu} = d_k\widetilde{\mu}(A_q) = \frac{d_k}{m} >0\]
for all $k \ge 1$. 
\end{proof}

\begin{thm}\label{second_inverse_thm_trees}
If $T$ is a tree and $2 \le r \le q$ with 
\begin{align}
\label{inverse_thm_ineq_1}
\underline{d}(G^\ast(F^r)) =\frac{\dim T - \log_q(r-1)}{1-\log_q(r-1)}>0
\end{align} 
or
\begin{align}
\label{inverse_thm_ineq_2}
\overline{d}(G^\ast(D^{r,2})) =\frac{\dim T - \log_q(r-1)}{1-\log_q(r-1)}>0,
\end{align}
then $\dim T = m^{-1}(1-\log_q(r-1))+\log_q(r-1)$ for some positive integer $m$. 
Moreover, $d^\ast(V_1^{r,m,mk}) >0$ for every $k \ge 1$. 
\end{thm}

\begin{proof}

Fix $\epsilon>0$ small enough, and let $\mathcal{R} = \{n \in \mathbb{N} \colon \widetilde{\eta_\epsilon}(A_r \cap S^{-n}A_r) >0\}$. 
In the case of (\ref{inverse_thm_ineq_1}), from the proof of Theorem \ref{Theorem4.1} we have 
\[\frac{\dim T - \log_q(r-1)}{1-\log_q(r-1)} = \underline{d}(G^\ast(F^r)) \ge \underline{d}(\mathcal{R}) \ge \widetilde{\eta_\epsilon}(A_r) \ge \frac{\dim T - \epsilon -\log_q(r-1)}{1-\log_q(r-1)},\]
so $\underline{d}(\mathcal{R}) \le \frac{3}{2}\widetilde{\eta_\epsilon}(A_r)$ for small enough $\epsilon$. 
By Theorem \ref{second_inverse_thm_measure} there is a positive integer $m$ such that $\mathcal{R}=m\mathbb{N}$ and $\widetilde{M} =\bigsqcup_{i=0}^{m-1} S^{-i} \left( \bigcup_{j=0}^{\infty} S^{-mj}A_r \right) $ up to $\widetilde{\eta_\epsilon}$-null sets. 

In the case of (\ref{inverse_thm_ineq_2}), we invoke the proof of Theorem \ref{direct-D-upper}. 
Recall that there exist a measurable set $Z$ such that $A_r \subset Z$ modulo $\widetilde{\eta_{\epsilon}}$-null sets and an increasing chain of measurable sets $(Z_\rho)_{\rho>0}$ with $\bigcup_{\rho > 0} Z_\rho = Z$ such that for every $\delta > 0$ we have 
\begin{align*}
\frac{\dim T - \log_q(r-1)}{1-\log_q(r-1)} &= \overline{d}(G^\ast(D^{r,2})) \ge \overline{d}(\mathcal{R}^\delta(Z_\rho)) \\
&\ge \left( \liminf_{N \to \infty} \frac{1}{N+1} \sum_{n=0}^N \frac{\widetilde{\eta_\epsilon}(Z_\rho \cap S^{-n}Z_\rho)}{\widetilde{\eta_\epsilon}(Z_\rho)} \right) - \delta\widetilde{\eta_\epsilon}(Z_\rho) \\
&\ge (1-\delta)\widetilde{\eta_\epsilon}(Z_\rho) \xrightarrow{\delta \to 0} \widetilde{\eta_\epsilon}(Z_\rho) \xrightarrow{\rho \to 0} \widetilde{\eta_\epsilon}(Z) \\
&\ge \widetilde{\eta_\epsilon}(A_r) \ge \frac{\dim T - \epsilon-\log_q(r-1)}{1-\log_q(r-1)},
\end{align*}
where $\mathcal{R}^\delta(Z_\rho) = \{ n \in \mathbb{N} \colon \widetilde{\eta_\epsilon}(Z_\rho \cap S^{-n}Z_\rho) > \delta\widetilde{\eta_\epsilon}(Z_\rho)^2\}$. 
Hence for small enough $\epsilon$ and $\rho$ the assumptions of Theorem \ref{Appendix_thm} are satisfied, so there exists $m \ge 1$ such that $\mathcal{R}(Z_\rho) = \mathcal{R}^\delta(Z_\rho) = m \mathbb{N}$, where $\mathcal{R}(Z_\rho) = \{ n \in \mathbb{N} \colon \widetilde{\eta_\epsilon}(Z_\rho \cap S^{-n}Z_\rho) > 0 \}$. 
Since this is true for all $\rho > 0$ small enough and $\mathcal{R} \subset \bigcup_{\rho > 0} \mathcal{R}(Z_\rho)$, we conclude that for $\epsilon$ small enough there exists $m \ge 1$ such that $\mathcal{R} \subset m \mathbb{N}$. 
This immediately implies that $\widetilde{M} =\bigsqcup_{i=0}^{m-1} S^{-i} \left( \bigcup_{j=0}^{\infty} S^{-mj}A_r \right) $ up to $\widetilde{\eta_\epsilon}$-null sets. 

In both cases, for small $\epsilon$ the above inequalities force 
\[\frac{\dim T - \log_q(r-1)}{1-\log_q(r-1)}  = m^{-1},\]
and hence
\begin{align}
\label{crucial}
\widetilde{\eta_\epsilon}(A_r) \ge (1 - \epsilon') \frac{1}{m} = (1- \epsilon') \widetilde{\eta_\epsilon}\left(\bigcup_{j=0}^{\infty} S^{-mj}A_r\right),
\end{align}
where $\epsilon' \to 0$ as $\epsilon \to 0$.
We also have 
\[\dim T \ge \widetilde{\eta_\epsilon}(A_r)+(1-\widetilde{\eta_\epsilon}(A_r))\log_q(r-1) \ge H(\widetilde{\eta_\epsilon}) \ge \dim T-\epsilon\] 
and hence the pair of inequalities
\begin{align}\label{A_r-entropy}
\int_{A_r} H\,d\widetilde{\eta_\epsilon}&\ge \widetilde{\eta_\epsilon}(A_r)-\epsilon \\ \label{A_r^c-entropy}
\int_{A_r^c} H\,d\widetilde{\eta_\epsilon} &\ge (1-\widetilde{\eta_\epsilon}(A_r))\log_q(r-1)-\epsilon.
\end{align}

We denote by $\bold{A_r} =  \bigcup_{j=0}^{\infty} S^{-mj}A_r$. 
Then we have $\widetilde{M} = \bigsqcup_{i=0}^{m-1} S^{-i}\bold{A_r}$.
Given $\widetilde{\tau} \in \widetilde{M}$ and $E \subset \widetilde{M}$, observe that $S^{-i}(\widetilde{\tau}) \subset E$ if and only if $P^i1_E(\widetilde{\tau}) = 1$. 
For $i \ge 0$, define $E_i$ to be $\bold{A_r}$ if $m$ divides $i$ and ${\bold{A_r^c}}$ otherwise. 
Then the $m$-periodicity of $\bold{A_r}$ and $\bold{A_r^c} =\bigsqcup_{i=1}^{m-1} S^{-i} \bold{A_r}$ under $S^{-1}$ gives $\widetilde{\eta_\epsilon}$-almost everywhere equalities
\[P^i1_{E_i} = P^iS^i1_{S^{-i}E_i} = 1_{S^{-i}E_i},\]
so the set $\bold{A_r'} = \bigcap_{i \ge 0} \{ \widetilde{\tau} \in \widetilde{M} \colon P^i1_{E_i}(\widetilde{\tau}) = 1\}$ is a $\widetilde{\eta_\epsilon}$-conull subset of $\bold{A_r}$. 

Define for $\delta >0$ the set
\[A_\delta = \bigcap_{i=0}^{mk} \{ \widetilde{\tau} \in \bold{A_r'} \colon P^iH(\widetilde{\tau}) \ge c_i-\delta\}, \qquad c_i = \begin{cases}
1 & \text{$m \mid i$}\\
\log_{q}(r-1) & \text{otherwise}. 
\end{cases}  \]
It follows from (\ref{crucial}) and inequalities (\ref{A_r-entropy}) and (\ref{A_r^c-entropy}) that by choosing  $\epsilon$ small enough we can guarantee that $\widetilde{\eta_\epsilon}(A_\delta) > 0$.
We will show the existence of $\delta$ such that the configuration $V_1^{r,m,mk}$ appears at the root of $|\tau_0|$ for every $\widetilde{\tau} = (\tau_i)_{i \le 0} \in A_\delta$. 
First notice that by construction of $\bold{A_r'}$, if $\widetilde{\tau} \in A_\delta$ and $v \in |\tau_0|$ with $0 \le l(v) \le mk$ then $H(\widetilde{\tau}^v) \le c_{l(v)}$. 
Hence for $0 \le i \le mk$
\begin{equation}
\label{P-entropy}
c_i-\delta \le P^iH(\widetilde{\tau}) = \sum_{l(v) = i} \tau_0(v)H(\widetilde{\tau}^v) \le c_i. 
\end{equation}
If $H(\widetilde{\tau}) > \log_q(q-1)$ then $\widetilde{\tau} \in A_q$, and if $H(\widetilde{\tau}) > \log_q(r-2)$ then $\widetilde{\tau} \in A_{r-1}$. 
To prove the appearance of $V_1^{r,m,mk}$ at the root of $|\tau_0|$ it therefore suffices to give sufficiently large lower bounds for $H(\widetilde{\tau}^v)$ for $l(v) \le mk$. 

\begin{lemma}
For every $\delta_1,\delta_2 >0$ there exists $\delta >0$ such that for $1 \le j \le mk+1$ (a) the set $\{\tau_0(v) \colon \widetilde{\tau} \in A_\delta, v \in |\tau_0|(j)\} \subset [0,1]$ is contained in an interval of length~$< \delta_1$, and (b) for all $\widetilde{\tau} \in A_\delta$ and $v \in |\tau_0|$ with $l(v) \le j-1$ we have $H(\widetilde{\tau}^v) \ge c_{l(v)}-\delta_2$. 
\end{lemma}

\begin{proof}
We prove both statements simultaneously by induction on $j$. 
For $j=1$ we have $H(\widetilde{\tau}) \ge 1-\delta$ for all $\widetilde{\tau} \in A_\delta$ by~(\ref{P-entropy}), so any $\delta < \delta_2$ suffices. 
Further, observe that $H$ is a continuous function attaining its maximum at $\widetilde{\tau}$ such that $p_{\widetilde{\tau}}(B_i) =q^{-1}$ for all $i \in \Lambda$. 
Hence given $\delta_1>0$ the set $\{\tau_0(v) \colon \widetilde{\tau} \in A_\delta, v \in |\tau_0|(1)\}$ is contained in an interval of length $<\delta_1$ (containing $q^{-1}$) for $\delta$ small enough. 

Assuming the lemma is true for $j \le i < mk+1$, we prove it for $j=i+1$. 
We first consider (b). 
For $w \in |\tau_0|(i)$ with $\widetilde{\tau} \in A_\delta$ the inequality (\ref{P-entropy}) gives 
\[c_i - \delta \le P^iH(\widetilde{\tau}) = \sum_{l(v)=i} \tau_0(v)H(\widetilde{\tau}^v) \le \tau_0(w)H(\widetilde{\tau}^w) + (1-\tau_0(w)) c_i,\]
and rearranging gives
\[c_i - \frac{\delta}{\tau_0(w)} \le H(\widetilde{\tau}^w). \]
By statement (a) of the induction hypothesis 
\[\sup_{\widetilde{\tau} \in A_\delta, w \in |\tau_0|(i)} \frac{\delta}{\tau_0(w)} \to 0\] as $\delta \to 0$, so by taking $\delta$ small enough statement (b) is satisfied for $j=i+1$. 
Combining statement (a) for $j=i$ and statement (b) for $j=i+1$ with the same argument as in the base case proves statement (a), noting that if $m$ does not divide $j$ then we consider maxima of $H$ on $A_r^c$. 
\end{proof}

It follows that any $V_1^{r,m,mk}$-detecting function is positive on $A_\delta$. 
By the correspondence principle Theorem \ref{ergodic-correspondence} we have for all $\epsilon >0$
\[d^\ast(V_1^{r,m,mk}) \ge \int_{\widetilde{M}} \phi_{V_1^{r,m,mk}}\,d\widetilde{\eta_\epsilon} \ge \int_{A_\delta} \phi_{V_1^{r,m,mk}}\,d\widetilde{\eta_\epsilon} >0,\]
since $\widetilde{\eta_\epsilon}(A_\delta)>0$. 
\end{proof}

\begin{thm}\label{third_inverse_thm_trees}
Let $\beta < 3/2$ and assume that $0 < \underline{d}(G^\ast(F^r)) < \beta \cdot\frac{\dim T - \log_q(r-1)}{1-\log_q(r-1)}$. Then there exists an integer $m \ge 1$ such that $m \mathbb{N} \subset G^\ast(F^r)$.
\end{thm}

\begin{proof}

For small enough $\epsilon>0$, by the correspondence principle Theorem \ref{ergodic-correspondence} and the proof of Theorem \ref{Theorem4.1}
\[\beta\widetilde{\eta_\epsilon}(A_r) \ge \beta \frac{\dim T-\epsilon-\log_q(r-1)}{1-\log_q(r-1)} >\underline{d}(G^\ast(F^r)) \ge \underline{d}(\mathcal{R}),\]
where $\mathcal{R} = \{ n \in \mathbb{N} \colon \widetilde{\eta_\epsilon}(A_r \cap S^{-n} A_r) > 0 \}$. 
Theorem \ref{second_inverse_thm_measure} then implies that there exists $m \ge 1$ such that $m\mathbb{N} = \mathcal{R} \subset G^\ast(F^r)$. 
\end{proof}

\begin{question}
It follows from the work of Furstenberg and Weiss in \cite{FW} that for every $k$ there exists~$n$ such that $d^\ast(D^{2,k}_n)>0$ provided that $\dim T>0$. 
On the other hand, under the assumptions of Theorem \ref{third_inverse_thm_trees}, there exists $m\ge 1$ such that $m \mathbb{N} \subset G^{\ast}(F)$. 
In analogy to Proposition~\ref{second_inverse_thm_sets}, is it true that the stronger claim $d^\ast(D^{2,k}_m) > 0$ holds true for every $k$ satisfying $(1-\beta^{-1})k < 1$?
\end{question}

\appendix

\section{Stability in inverse theorem \ref{second_inverse_thm_measure}}\label{appendix}

In the proof of Theorem \ref{second_inverse_thm_trees} for the configuration $D^{r,2}$, we are unable to apply Theorem \ref{second_inverse_thm_measure} since we have no upper bound for $\underline{d}(\mathcal{R})$. 
However, we have bounds on the densities of the sets of $\delta$-return times. 
Here we prove a stability result (Theorem \ref{Appendix_thm}) giving the same conclusion as Theorem \ref{second_inverse_thm_measure} under assumptions involving $\delta$-return times instead of return times. 

Given an ergodic measure-preserving system $(X,\mathscr{B},\nu,S)$ and $A \in \mathscr{B}$ with $\nu(A)>0$, define for $\delta >0$ the set of $\delta$-return times of $A$
\[
\mathcal{R}^{\delta} = \{ n \in \mathbb{N} \colon \nu(A \cap S^{-n}A) > \delta \nu(A)^2 \}.
\]
Define also for $0<\gamma<1$ the set
\[\mathcal{R}_\gamma = \{n \in \mathbb{N} \colon \nu(A \cap S^{-n}A) \ge (1-\gamma)\nu(A)\}.\]

\begin{lemma}
\label{estimate_lemma}
If $\overline{d}(\mathcal{R}^\delta) \le (1+\eta)\nu(A)$ for all $\delta>0$, then for any $\gamma>0$
\[\underline{d}(\mathcal{R}_\gamma) \ge \left(\frac{\gamma-\eta+\gamma\eta}{\gamma}\right)\nu(A).\]
\end{lemma}

\begin{proof}
Given $\gamma$, choose $\delta$ such that $0 < \delta < \frac{1-\gamma}{\nu(A)}$ (so $\mathcal{R}_\gamma \subset \mathcal{R}^\delta$). 
First observe that by the mean ergodic theorem
\begin{equation}
\label{delta_nuA}
\underline{d}(\mathcal{R}^\delta) = \liminf_{N \to \infty} \frac{1}{N+1} \sum_{n=0}^N 1_{\mathcal{R}^\delta}(n) \ge \liminf_{N \to \infty} \frac{1}{N+1} \sum_{n=0}^N \frac{\nu(A \cap S^{-n}A)}{\nu(A)}-\delta\nu(A) = (1-\delta)\nu(A).
\end{equation}

Let $(N_k)_{k \ge 1}$ be an increasing sequence such that $\underline{d}(\mathcal{R}_\gamma) = \lim_{k \to \infty} \frac{1}{N_k+1} \sum_{n=0}^{N_k}1_{\mathcal{R}_\gamma}(n)$. 
By the mean ergodic theorem
\begin{align*}
\nu(A) &= \lim_{k \to \infty} \frac{1}{N_k+1} \sum_{n=0}^{N_k}\frac{\nu(A \cap S^{-n}A)}{\nu(A)} \\
&\le \limsup_{k \to \infty} \frac{1}{N_k+1}\sum_{\substack{n \le N_k \\ n \in \mathcal{R}_\gamma}} \frac{\nu(A \cap S^{-n}A)}{\nu(A)} + \limsup_{k \to \infty} \frac{1}{N_k+1}\sum_{\substack{n \le N_k \\ n \in \mathcal{R}^\delta \setminus \mathcal{R}_\gamma}} (1-\gamma) + \limsup_{k \to \infty} \frac{1}{N_k+1} \sum_{\substack{n \le N_k \\ n \in (\mathcal{R}^\delta)^c}} \delta\nu(A) \\
&\le \underline{d}(\mathcal{R}_\gamma)+(1-\gamma)\left(\overline{d}(\mathcal{R}^\delta)-\underline{d}(\mathcal{R}_\gamma)\right) + \delta\nu(A)(1-\underline{d}(\mathcal{R}^\delta)) \\
&\le \gamma\underline{d}(\mathcal{R}_\gamma) + (1-\gamma)(1+\eta)\nu(A) + \delta\nu(A)(1-(1-\delta)\nu(A))
\end{align*}
where in the last inequality we used the assumption $\overline{d}(\mathcal{R}^\delta) \le (1+\eta)\nu(A)$ and (\ref{delta_nuA}). 
Rearranging, we obtain
\[\underline{d}(\mathcal{R}_\gamma) \ge \left(\frac{\gamma-\eta+\gamma\eta-\delta+\delta\nu(A)-\delta^2\nu(A)}{\gamma}\right)\nu(A).\]
Taking $\delta \to 0$ gives the required inequality. 
\end{proof}

For $l \in \mathbb{N}$ and $\delta>0$, define the set
\[\mathcal{R}_l^\delta = \{n \in \mathbb{N} \colon \nu(A \cap S^{-n}A \cap S^{-(l+n)}A) > \delta\nu(A)^3\}.\]

\begin{lemma}
\label{erg_thm_recurrence}
If $l \in \mathcal{R}^{\delta}$, then $\underline{d}(\mathcal{R}^{\delta\varepsilon}_l) \ge (1-\varepsilon)\nu(A)$ for all $\varepsilon>0$. 
\end{lemma}

\begin{proof}
Given $\varepsilon >0$ and $A,B \in \mathscr{B}$ of positive measure, the set of $\varepsilon$-transfer times from $A$ to $B$ is $\mathcal{R}_{A,B}^{\varepsilon} = \{ n \in \mathbb{N} \colon \nu(A \cap S^{-n}B) > \varepsilon \nu(A) \nu(B)\}$.
Observe that 
\[
\mathcal{R}^\varepsilon_{A,A\cap S^{-l}A} = \{ n \in \mathbb{N} \colon \nu(A \cap S^{-n}(A \cap S^{-l}A)) > \varepsilon \nu(A) \nu(A \cap S^{-l}A)\} \subset \mathcal{R}_{l}^{\delta \varepsilon}.
\]
By the mean ergodic theorem,
\begin{align*}
\label{mean-erg}
\underline{d}(\mathcal{R}_{A,B}^{\varepsilon}) &= \liminf_{N \to \infty} \frac{1}{N+1} \sum_{n=0}^N 1_{\mathcal{R}_{A,B}^{\varepsilon}}(n)  \\
&\ge \liminf_{N \to \infty} \frac{1}{N+1} \sum_{n=0}^N \frac{\nu(A \cap S^{-n}B)}{\min{(\nu(A),\nu(B))}} - \varepsilon \max
{(\nu(A),\nu(B))} \\
&\ge (1 - \varepsilon) \max{(\nu(A),\nu(B))},
\end{align*}
so we have
\[
\underline{d}(\mathcal{R}_{l}^{\delta \eps}) \ge \underline{d}(\mathcal{R}^\varepsilon_{A,A \cap S^{-l}A}) \ge (1 - \varepsilon) \nu(A)
\]
as required. 
\end{proof}

\begin{thm}
\label{Appendix_thm}
If for $\eta<\frac{1}{5}$ we have $\overline{d}(\mathcal{R}^{\delta}) \le (1+\eta) \nu(A)$ for every $\delta > 0$, then there exists $m \ge 1$ such that $\mathcal{R}^\delta = m \mathbb{N}$ for all sufficiently small $\delta$. 
\end{thm}

\begin{proof}

Fix $0 < \eta < \frac{1}{5}$ such that $\overline{d}(\mathcal{R}^{\delta}) \le (1+\eta) \nu(A)$, and choose $\gamma$ satisfying 
\begin{equation}
\label{gamma_bd}
\frac{3\eta}{1+\eta} < \gamma < \frac{1}{2}.
\end{equation} 
Observe that $\mathcal{R}_\gamma + \mathcal{R}_\gamma \subset \mathcal{R}_{2\gamma} \subset \mathcal{R}^\delta$ for $0 < \delta < \frac{1-2\gamma}{\nu(A)}$ by Lemma \ref{lemma4}. 
Noting that (\ref{gamma_bd}) implies $\gamma-\eta+\gamma\eta>0$ and $\frac{(1+\eta)\gamma}{\gamma-\eta+\gamma\eta} < 2$, Lemma \ref{estimate_lemma} gives
\begin{equation}
\label{eq_complicated}
\underline{d}(\mathcal{R}_\gamma+\mathcal{R}_\gamma) \le \underline{d}(\mathcal{R}^\delta) \le \overline{d}(\mathcal{R}^\delta) \le (1+\eta)\nu(A) \le \left(\frac{(1+\eta)\gamma}{\gamma-\eta+\gamma\eta}\right)\underline{d}(\mathcal{R}_\gamma) < 2\underline{d}(\mathcal{R}_\gamma).
\end{equation}
Kneser's theorem then gives the existence of an integer $m \ge 1$ and $K \subset \{0,1,\ldots,m-1\}$ such that 
\begin{itemize}
\item 
$\mathcal{R}_\gamma \subset K + m \mathbb{N}$,
\item 
$|K+K| = 2 |K| -1$, where the operation on the left hand side is in $ \mathbb{Z} /m \mathbb{Z}$, and
\item  
$\mathcal{R}_\gamma + \mathcal{R}_\gamma \subset K+K +m\bN$ with $ |(K+K +m\bN) \setminus (\mathcal{R}_\gamma + \mathcal{R}_\gamma)| < \infty$.
\end{itemize}
Combining this with Lemma \ref{estimate_lemma} gives
\begin{multline*}
\frac{2|K|-1}{m} = \frac{|K+K|}{m} \le \overline{d}(\mathcal{R}_\gamma+\mathcal{R}_\gamma) \le \overline{d}(\mathcal{R}^\delta) \le (1+\eta)\nu(A) \\
\le \left(\frac{(1+\eta)\gamma}{\gamma-\eta+\gamma\eta}\right)\underline{d}(\mathcal{R}_\gamma) \le \left(\frac{(1+\eta)\gamma}{\gamma-\eta+\gamma\eta}\right)\frac{|K|}{m},
\end{multline*}
and rearranging gives
\[|K| \le 1+\frac{\eta}{\gamma+\gamma\eta-2\eta} <2,\]
where the last inequality follows from (\ref{gamma_bd}). 
Hence $|K|=1$. 
Furthermore $K=\{0\}$, since otherwise $\mathcal{R}_\gamma$ and $\mathcal{R}_\gamma+\mathcal{R}_\gamma$ would be disjoint subsets of $\mathcal{R}^\delta$ giving the contradiction
\[\overline{d}(\mathcal{R}^\delta) \ge \underline{d}(\mathcal{R}_\gamma+\mathcal{R}_\gamma)+\underline{d}(\mathcal{R}_\gamma) \ge 2\underline{d}(\mathcal{R}_\gamma) > \overline{d}(\mathcal{R}^\delta).\]

We first prove that $\mathcal{R}^\delta \subset m\mathbb{N}$ for small enough $\delta>0$. 
For $l \in \mathbb{N}$ and $\delta>0$, recall 
\[\mathcal{R}_l^{\delta^2} = \{n \in \mathbb{N} \colon \nu(A \cap S^{-n}A \cap S^{-(l+n)}A) > \delta^2\nu(A)^3\}.\]
Since $\nu(A \cap S^{-(l+n)}A) \ge \nu(A \cap S^{-n}A \cap S^{-(l+n)}A)$, if $n \in \mathcal{R}_l^{\delta^2}$ then $l+n \in \mathcal{R}^{\delta^2\nu(A)}$. 
Assuming $l \in \mathcal{R}^\delta \setminus m\mathbb{N}$, we derive a contradiction. 
Observe that $\mathcal{R}_\gamma+\mathcal{R}_\gamma \subset \mathcal{R}^\delta \subset \mathcal{R}^{\delta^2\nu(A)}$, so
\begin{align*}
\overline{d}(\mathcal{R}^{\delta^2\nu(A)}) &\ge \underline{d}(\mathcal{R}_\gamma+\mathcal{R}_\gamma)+\underline{d}((l+\mathcal{R}^{\delta^2}_l) \setminus m\mathbb{N}) \\
&\ge m^{-1} + \underline{d}(l+(\mathcal{R}_l^{\delta^2} \cap m\mathbb{N})) \\
&= m^{-1}+\underline{d}(\mathcal{R}_l^{\delta^2} \cap m\mathbb{N}),
\end{align*}
where the second inequality uses the assumption on $l$. 
Since $\mathcal{R}_l^{\delta^2}, m\mathbb{N} \subset \mathcal{R}^{\delta^2\nu(A)}$ (up to a finite set), by Lemma \ref{erg_thm_recurrence}
\begin{align*}
\underline{d}(\mathcal{R}_l^{\delta^2} \cap m\mathbb{N}) &\ge \underline{d}(\mathcal{R}_l^{\delta^2})+\underline{d}(m\mathbb{N}) - \overline{d}(\mathcal{R}_l^{\delta^2} \cup m\mathbb{N}) \\
&\ge (1-\delta)\nu(A) + m^{-1} - \overline{d}(\mathcal{R}^{\delta^2\nu(A)}). 
\end{align*}
Using the hypothesis $\overline{d}(\mathcal{R}^{\delta^2\nu(A)}) \le (1+\eta)\nu(A)$ we obtain 
\begin{equation}
\label{delta_nu_A}
2(1+\eta)\nu(A) \ge 2\overline{d}(\mathcal{R}^{\delta^2\nu(A)}) \ge (1-\delta)\nu(A)+2m^{-1}.
\end{equation}
Since $|K|=1$, Kneser's theorem and Lemma \ref{estimate_lemma} imply 
\[m^{-1} \ge \underline{d}(\mathcal{R}_\gamma) \ge \left(\frac{\gamma-\eta+\gamma\eta}{\gamma}\right)\nu(A),\]
and combining with (\ref{delta_nu_A}) gives $\gamma \le \frac{2\eta}{1-\delta}$ after rearranging. 
This is compatible with (\ref{gamma_bd}) only if $\eta >\frac{1-3\delta}{2}$. 
Since $\eta < \frac{1}{5}$, it follows that the above requires $\delta>\frac{1}{5}$. 
Hence $l \in \mathcal{R}^\delta \setminus m\mathbb{N}$ gives a contradiction and $\mathcal{R}^\delta \subset m\mathbb{N}$ for small enough $\delta>0$.  

Finally we show $\mathcal{R}^{\delta} = m \mathbb{N}$ for small $\delta$. 
Indeed, since $\underline{d}(\mathcal{R}_\gamma) > \frac{1}{2m}$ by combining equation (\ref{eq_complicated}) with the third implication of Kneser's theorem $R_{\gamma} +R_{\gamma} \subset K +K + m \mathbb{N}$ and the fact that $|K| = 1$, for every $l \in \mathbb{N}$ there exists $n \in \mathbb{N}$ such that $mn,m(n+l) \in \mathcal{R}_{\gamma}$.
Therefore 
\begin{align*}
\nu(A \cap S^{-ml}A) &= \nu(S^{-mn}A \cap S^{-m(n+l)}A) \\
&\ge \nu((A \cap S^{-mn}A) \cap (A \cap S^{-m(n+l)}A)) \\
&\ge \nu(A \cap S^{-mn}A) + \nu(A \cap S^{-m(n+l)}A) - \nu(A) \\
&\ge (1-2\gamma)\nu(A) > \delta\nu(A)^2
\end{align*}
for $\delta < \frac{1-2\gamma}{\nu(A)}$, so for sufficiently small $\delta>0$ we have $ml \in \mathcal{R}^\delta$ for all $l \in \mathbb{N}$. 
\end{proof}

\subsection*{Discussion}

The set of transfer times $\mathcal{R}_{A,B}$ has strong parallels with the difference set 
$\mathcal{A}-\mathcal{B} = \{ a-b \colon a\in \mathcal{A},\, b\in \mathcal{B} \}$, $\mathcal{A}, \mathcal{B} \subset \bZ / r \bZ$, which is one of the main objects of Additive Combinatorics. 
For example, the lower bound for $\underline{d}(\mathcal{R}^\varepsilon_{A,B})$ in  Lemma \ref{erg_thm_recurrence} corresponds to the simple fact that $|\mathcal{A} - \mathcal{B}| \ge \max\{|\mathcal{A}|, |\mathcal{B}| \}$.
%The last bound is known to be tight 
It is easy to see that the bound is tight 
%and the counterexample is   
%Indeed 
and is attained when 
$\mathcal{B} - \mathcal{B}$ belongs to the centraliser of $\mathcal{A}$ (or vice versa). 
It implies that $\mathcal{A}$ and $\mathcal{B}$ have some periodic structure and it is analogous to our conclusions in Theorems \ref{second_inverse_thm_measure} and \ref{Appendix_thm} on the structure of our dynamical system.
On the other hand, if $\mathcal{A} = \{0,1\} \subseteq \bZ/r\bZ$ for large $r$, then $\mathcal{A} - \mathcal{A} = \{0,1,-1\}$ and hence $\eta$ in Theorem \ref{Appendix_thm} must be less than $1/2$.
Moreover, the sets $\mathcal{R}_m^\delta$ from Lemma \ref{erg_thm_recurrence} which are used in the proof of Theorem \ref{Appendix_thm} can be thought as a dynamical version of the well--known combinatorial  $e$--transform, see, e.g., \cite[Section 5.1]{TV}.
Although it is non--obvious how to define the higher sumsets in the dynamical context, an analogue of the Pl\"unnecke--Rusza triangle inequality for dynamical systems would be a first step towards such a theory.

\begin{question}
Assume that $(X,\mathscr{B},\nu,S)$ is an invertible ergodic system and $d(\mathcal{R}_{A,B})$, $d(\mathcal{R}_{A,C})$, $d(\mathcal{R}_{B,C})$ exist for $A,B,C \in \mathscr{B}$.
Is it true that 
\[\mu(C) d(\mathcal{R}_{A,B}) \le d(\mathcal{R}_{A,C}) d(\mathcal{R}_{B,C}) \,?\]
\end{question}

\end{document}